\documentclass[a4paper,11pt]{amsbook}
\usepackage{amsfonts}
\usepackage{amssymb}
\usepackage{xspace}
\usepackage[utf8]{inputenc}
\usepackage[T1]{fontenc}
\usepackage{enumitem}
\usepackage[ruled,titlenumbered,linesnumbered,noend]{algorithm2e}
\usepackage{todonotes}
\usepackage[pdfborder={0 0 0},colorlinks=false]{hyperref}
\usepackage{tocvsec2}

\theoremstyle{plain}

\newtheorem{theorem}{\textbf{Theorem}}[chapter]
\newtheorem{question}[theorem]{Question}
\newtheorem{case1}{Case}
\newtheorem{case2}{Case}
\newtheorem{claim}[theorem]{Claim}
\newtheorem{conjecture}[theorem]{\textbf{Conjecture}}
\newtheorem{corollary}[theorem]{Corollary}
\newtheorem{definition}[theorem]{Definition}
\newtheorem{example}[theorem]{Example}
\newtheorem{lemma}[theorem]{Lemma}
\newtheorem{proposition}[theorem]{Proposition}
\newtheorem{remark}[theorem]{Remark}
\numberwithin{equation}{section}

\newcommand{\field}[1]{\mathbb{#1}}
\newcommand{\R}{\field{R}}

\newcommand{\N}{\field{N}}
\newcommand{\Q}{\field{Q}}
\newcommand{\Le}{\field{L}}
\newcommand{\Z}{\field{Z}}
\newcommand{\K}{\field{K}}

\DeclareMathOperator{\cl}{cl}
\DeclareMathOperator{\trdeg}{trdeg}
\DeclareMathOperator{\spam}{span}

\DeclareMathOperator{\ch}{ch}
\DeclareMathOperator{\proj}{Proj}
\DeclareMathOperator{\lk}{lk}

\begin{document}
\author{Michał Lasoń}

\thispagestyle{empty}
\centerline{}
\vspace{6cm}
\centerline{\Huge\bf Coloring Games and Algebraic Problems}
\vspace{0.2cm}
\centerline{\Huge\bf on Matroids}
\vspace{1cm}
\centerline{\huge Michał Lasoń}
\vspace{2cm}
\centerline{\large Ph.D. thesis}
\centerline{\large Theoretical Computer Science Department}
\centerline{\large Faculty of Mathematics and Computer Science}
\centerline{\large Jagiellonian University}
\centerline{\large advised by Jarosław Grytczuk}
\vspace{7.9cm}
\centerline{\large Kraków 2013}

\cleardoublepage

\tableofcontents

\chapter*{Introduction}

This thesis is basically devoted to matroids --- fundamental structure of combinatorial optimization --- though some of our results concern simplicial complexes, or Euclidean spaces. We study old and new problems for these structures, with combinatorial, algebraic, and topological flavor. Therefore the thesis splits naturally into three parts, accordingly to these three aspects.\bigskip

In the first part we study several coloring games on matroids. Proper colorings of a ground set of a matroid (colorings in which every color class forms an independent set) were already studied by Edmonds \cite{Ed65}. He showed an explicit formula for the chromatic number of a matroid (the least number of colors in a proper coloring) in terms of rank function. We generalize (Theorem \ref{TheoremGeneralizationOfSeymour}) a theorem of Seymour \cite{Se98} on the list chromatic number of a matroid (another important result concerning colorings). 

Connections between matroids and games go back to the famous Shannon switching game (cf. \cite{Mi61}), invented independently by Gale (cf. \cite{We01}). Its matroidal version, introduced by Lehman \cite{Le64}, was solved by Edmonds \cite{Ed66} already in 1960's. However our games are of a slightly different nature. 

For instance, in one variant two players are coloring a matroid properly, but only the first player is interested in finishing the job (with a given set of colors). The second, `bad' one, tries to prevent it from happening. A natural question arises: how many colors are needed for the `good guy' to win, compared to the chromatic number of a matroid. We show that this parameter is at most twice as big as the chromatic number (Theorem \ref{TheoremGeneralizedGame}). This improves and extends to arbitrary matroids the result of Bartnicki, Grytczuk and Kierstead \cite{BaGrKi08}, which concerns only graphic matroids. Moreover, we show that this bound is in general almost optimal (up to at most one).

In another game, we consider a coloring of a matroid from lists, but in a situation when only part of information about colors in the lists is known. The `bad player' assigns consecutive colors to lists of some elements arbitrarily, while the `good' one colors properly elements, each with a color from its list immediately after the color is assigned. It is called the on-line list coloring game. We prove, with Wojciech Lubawski, that the chromatic parameter of this game is the same as its off-line version (Theorem \ref{TheoremOnLineGeneralizationOfSeymour}). This generalizes a theorem of Seymour \cite{Se98} to the on-line setting. 

Games of this type were initially investigated for graphs, leading to lots of interesting results and fascinating conjectures (cf. \cite{Ga81,Bo91,BaGrKiZh07} and \cite{Sc09,Zh09} respectively). We hope that our results add a new aspect of structural type, by showing that most of pathological phenomena appearing for graph coloring games are no longer present in the realm of matroids.\bigskip

The second part of the thesis studies two problems of algebraic nature. The first one concerns $f$-vectors of simplicial complexes. These vectors encode the number of faces of a given size in the complex, and are characterized in the celebrated Kruskal-Katona inequality \cite{Kr63,Ka68}. A similar characterization for matroids or, more generally, for pure complexes (those with all maximal faces of the same size) remains elusive. Our main result (Theorem \ref{TheoremExtremalAreDecomposable}) asserts that every extremal simplicial complex (pure, with equality in the Kruskal-Katona inequality) is vertex decomposable. This extends a theorem of Herzog and Hibi \cite{HeHi99} asserting that extremal complexes have the Cohen-Macaulay property (it is a simple fact that the last property is implied by vertex decomposability). Our argument is purely combinatorial with the main inspiration coming from the proof of the Kruskal-Katona inequality.

The second problem we consider in this part of the thesis is a long standing conjecture of White \cite{Wh80} from 1980 (Conjecture \ref{ConjectureWhite}). The problem concerns the symmetric exchange phenomenon in matroids, and it is solved only for some very special classes of matroids (cf. \cite{Bl08,Bo13,Ka10, Sc11,Co07}). In a simplest form White's conjecture says that if two families of bases of a matroid have equal unions (as multisets), then one can pass between them by a sequence of single element symmetric exchanges. In the algebraic language this means that the toric ideal of a matroid is equal to the ideal generated by quadratic binomials corresponding to symmetric exchanges. In a joint work with Mateusz Michałek we prove White's conjecture up to saturation (Theorem \ref{TheoremMain2}), that is that the saturations of both ideals are equal. We believe that it is the first result in this direction valid for arbitrary matroids. Additionally we prove the full conjecture for a new large class -- strongly base orderable matroids (Theorem \ref{TheoremMain1}).\bigskip

The last part of the thesis concerns the famous necklace splitting problem. Suppose we are given a colored necklace (a segment of integers, or a line segment), and we want to cut it so that the resulting pieces can be fairly split into two (or more) parts, which means that each part captures the same amount of every color. A theorem of Goldberg and West \cite{GoWe85} asserts that if the number of parts is two, then every $k$-colored necklace has a fair splitting using at most $k$ cuts. This fact has a simple proof \cite{AlWe86} using the celebrated Borsuk-Ulam theorem (cf. \cite{Ma03}). Alon \cite{Al87} extended this result to arbitrary number of parts by showing that $k(q-1)$ cuts suffice in the case of $q$ parts.

In a joint paper with Alon, Grytczuk, and Michałek \cite{AlGrLaMi09} we studied a kind of opposite question, which was motivated by the problem of Erd\H{o}s on strongly nonrepetitive sequences \cite{Er61} (cf. \cite{Cu93,De79,Gr08}). We proved that for every $k$ there is a $(k+3)$-coloring of the real line such that no segment has a fair splitting into two parts with at most $k$ cuts. The main results of this part of the thesis generalize this theorem to arbitrary dimension and arbitrary number of parts. We consider two versions -- with cuts made by axis-aligned hyperplanes (Theorem \ref{TheoremNecklaces}), and by arbitrary hyperplanes (Theorem \ref{TheoremNecklacesArbitrary}). In the first case the upper bound we achieve almost matches the lower bound implied by a theorem of Alon \cite{Al87} (or by a more general theorem of de Longueville and \v{Z}ivaljevi\v{c} \cite{LoZi08}). The methods we use relay on the topological Baire category theorem and several applications of algebraic matroids.

\settocdepth{part}
\chapter*{Acknowledgements}
\settocdepth{section}

Firstly, I would like to thank my advisor Jarosław Grytczuk for his scientific guidance during my entire Ph.D. studies. It was a great pleasure and opportunity to learn from him the taste of combinatorics. Throughout the thesis he provided me with encouragement and a lot of interesting problems. In particular he introduced me to the game chromatic number of a matroid. He also was the first person who asked if Seymour's theorem generalizes to the on-line setting. Finally, Jarek is a coauthor of my first scientific publication. I am sincerely grateful for his effort.\smallskip

I would like to express my gratitude to professor J\"{u}rgen Herzog for pointing a direction in which I can generalize one of his results, and also for showing me the conjecture of White.\smallskip

I am especially grateful to Wojciech Lubawski and Mateusz Michałek, who are coauthors of some of the results presented in this thesis. I thank them for inspiration, motivation and hours of fruitful discussions.\smallskip

I thank my friends: Łucja Farnik, Michał Farnik, Adam G\k{a}gol, Andrzej Grzesik, Grzegorz Gutowski, Jakub Kozik, Tomasz Krawczyk, Piotr Micek, Arkadiusz Pawlik, Małgorzata Sulkowska and Bartosz Walczak for interesting talks, questions, answers and a great working atmosphere.\smallskip

Writing this thesis would not be possible without the constant support of my family. I thank a lot my mum for her love and motivation.\bigskip

The author has been granted scholarship in the project \emph{Doctus Małopolski fundusz stypendialny dla doktorantów} co-funded by the European Union through the European Social Fund, and also supported by the grant of the Polish Ministry of Science and Higher Education \emph{On-line tasks selection}  N~N206 568240. 


\chapter{Introduction to Matroids}

Matroids are one of basic mathematical structures. They abstract the idea of independence from various areas of mathematics, ranging from algebra to combinatorics. The origins of matroid theory turn back to 1930's, when they were introduced by Whitney \cite{Wh35}. For an adequate introduction we refer the reader to the book of Oxley \cite{Ox92} (and to \cite{Ox03} for a sketchy one).

In this Chapter we give a brief sketch of basic concepts of matroid theory. We present definitions, examples and several classes of matroids. We also prove Matroid Union Theorem \ref{TheoremMatroidUnion}, a basic tool, which already reveals regularity of the structure of matroid. Therefore, this Chapter can be treated as preliminaries, in particular it does not contain any our result. 

\section{Definitions and Terminology}

Matroids are cryptomorphic structures, that is they can be defined in many different but nontrivially equivalent ways. We will introduce five of this axiom systems which are most important for us. We show correspondence between the first one, which can be considered as basic, and the others (for proofs see \cite[Chapter 1]{Ox92}). This will also provide a necessary notation.

Usually by a matroid $M$ on a finite ground set $E$ we mean a collection $\mathcal{I}$ of subsets of $E$ satisfying the following conditions.

\begin{definition}\emph{(Independent Sets)}\label{DefinitionIndependentSets}
A \emph{matroid} $M$ consists of a finite set $E$, called the \emph{ground set}, and a set $\mathcal{I}$ of \emph{independent subsets} of $E$ satisfying the following conditions:
\begin{enumerate}
\item empty set is an independent set,
\item subset of an independent set is independent,
\item if $I,J\in\mathcal{I}$ and $\left\vert I\right\vert<\left\vert J\right\vert$, then we can extend $I$ with an element of $J$, that is $I\cup e\in\mathcal{I}$ for some $e\in J\setminus I$ (augmentation axiom).
\end{enumerate}
\end{definition}

Actually, as we will see in Definition \ref{DefinitionRepresentableMatroid}, any finite subset $E$ of a vector space together with a family of linearly independent sets contained in $E$ constitutes a basic example of a matroid. This connection justifies a large part of terminology used in matroid theory. For example, one may define bases of a matroid, the rank of a subset, or the closure operation just the same way as for vector spaces.

\begin{definition}\emph{(Bases)}\label{DefinitionBases}
A \emph{matroid} $M$ consists of a finite set $E$, called the \emph{ground set}, and a set $\mathcal{B}$ of subsets of $E$, called \emph{bases}, satisfying the following conditions:
\begin{enumerate}
\item $\mathcal{B}$ is non-empty,
\item if $B_1,B_2$ are bases and $e\in B_1\setminus B_2$, then there is $f\in B_2\setminus B_1$, such that $B_1\cup f\setminus e$ is a basis.
\end{enumerate}
\end{definition}

Bases are just maximal independent sets. Due to augmentation axiom all of them have equal size, so bases are in fact maximum independent sets. The other direction is also clear, having the set of bases, independent sets are exactly those which are contained in some basis. 

\begin{definition}\emph{(Rank Function)}\label{DefinitionRank}
A \emph{matroid} $M$ consists of a finite set $E$, called the \emph{ground set}, and a function $r:\mathcal{P}(E)\rightarrow\N$, called the \emph{rank function}, satisfying the following conditions:
\begin{enumerate}
\item if $A\subset E$, then $0\leq r(A)\leq \left\vert A\right\vert$,
\item rank is monotone, that is if $A\subset B\subset E$, then $r(A)\leq r(B)$,
\item if $A,B\subset E$, then $r(A\cup B)+r(A\cap B)\leq r(A)+r(B)$ (submodularity axiom).
\end{enumerate}
\end{definition}

The rank of a set $A\subset E$ is the size of the maximal independent set contained in $A$. Knowing the rank function, independent sets are exactly whose sets $A\subset E$ for which $r(A)=\left\vert A\right\vert$. So as independent sets abstract among others linear independence or algebraic independence, the rank function abstracts dimension or transcendence degree. Sometimes the rank of the whole matroid $M$, which is $r(E)$, will be denoted by $r(M)$ in order not to refer to the ground set.

\begin{definition}\emph{(Closure)}\label{DefinitionClosure}
A \emph{matroid} $M$ consists of a finite set $E$, called the \emph{ground set}, and a \emph{closure operator} $\cl:\mathcal{P}(E)\rightarrow\mathcal{P}(E)$ satisfying the following conditions:
\begin{enumerate}
\item if $A\subset E$, then $A\subset\cl(A)$,
\item if $A\subset B\subset E$, then $\cl(A)\subset\cl(B)$,
\item if $A\subset E$, then $\cl(\cl(A))=\cl(A)$,
\item if $A\subset E,e\in E$, and $f\in\cl(A\cup e)\setminus\cl(A)$, then $e\in\cl(A\cup f)$.
\end{enumerate}
\end{definition}

The closure of a set $A\subset E$ is the set of all elements $e\in E$ satisfying $r(A\cup e)=r(A)$. Or, in other words, $\cl(A)$ is the maximum set $B$ containing $A$ for which $r(B)=r(A)$. Independent sets are exactly those sets which do not have a proper subset with the same closure.

Another source of matroid terminology comes from graph theory. This is because, as we will see in Definition \ref{DefinitionGraphicMatroid}, every graph $G$ can be turned into a graphic matroid $M(G)$ on the set of edges by defining $\mathcal{I}$ as the family of sets not containing a cycle. Therefore circuits in a matroid, which generalize cycles of graphic matroid, are minimal sets which are not independent. Thus, a set is independent if it does not contain any circuit. 

\begin{definition}\emph{(Circuits)}\label{DefinitionCircuits}
A \emph{matroid} $M$ consists of a finite set $E$, called the \emph{ground set}, and a set $\mathcal{C}$ of subsets of $E$, called \emph{circuits}, satisfying the following conditions:
\begin{enumerate}
\item empty set is not a circuit,
\item there is no inclusion between any two circuits,
\item if $C_1,C_2$ are two circuits and $e\in C_1\cap C_2$, then there is a circuit $C_3$ such that $C_3\subset C_1\cup C_2\setminus e$.
\end{enumerate}
\end{definition}

There are also several other ways to define a matroid, for example by flats or hyperplanes.

\section{Examples and Operations}\label{SectionExamplesAndOperations}

We show basic examples and classes of matroids (for proofs that they satisfy axioms of matroid see \cite{Ox92}). More examples can be produced using operations on matroids described in the second part of this Section. We begin with an almost trivial example. 

\begin{definition}\emph{(Uniform Matroid)}\label{DefinitionUniformMatroid} Let $E$ be a finite set, and let $b$ be an integer. Then, $$(E,\mathcal{I})=(E,\{A\subset E:\left\vert A\right\vert\leq b\})$$ is a matroid denoted by $U_{r(E),\left\vert E\right\vert}$, with the rank function $r(A)=\min\{\left\vert A\right\vert,b\}$.
\end{definition}

The next two examples come from algebra. The second will play a big role in the last part of the thesis.

\begin{definition}\emph{(Representable Matroid)}\label{DefinitionRepresentableMatroid} Let $V$ be a vector space over a field $\K$, and let $E\subset V$ be a finite set. Then, $$(E,\mathcal{I})=(E,\{A\subset E:A\text{ is linearly independent over }\K\})$$ is a matroid with the rank function $r(A)=\dim_{\K}(\spam(A))$ and closure operator $\cl(A)=\spam(A)\cap E$. Such a matroid is called \emph{representable over the field $\K$}. If a matroid is representable over some field it is called \emph{representable}.
\end{definition}

\begin{definition}\emph{(Algebraic Matroid)}\label{DefinitionAlgebraicMatroid} Let $\K\subset\Le$ be a field extension, and let $E\subset\Le$ be a finite set. Then, $$(E,\mathcal{I})=(E,\{A\subset E:A\text{ is algebraically independent over }\K\})$$ is a matroid with the rank function $r(A)=\trdeg_{\K}(\K(A))$. 
\end{definition}

The following example is probably the most natural for combinatorics. Its big advantage is that it can be easily visualized. We will often refer to graphic matroids in examples and some theorems. 

\begin{definition}\emph{(Graphic Matroid)}\label{DefinitionGraphicMatroid} Let $G=(V,E)$ be a multigraph. Then, $$(E,\mathcal{I})=(E,\{A\subset E:A\text{ does not contain a cycle}\})$$ is a matroid with the set of bases equal to the set of spanning forests of $G$, and the set of circuits equal to the set of cycles of $G$. We denote it by $M(G)$. 
\end{definition}

Graphic matroids are also interesting because they are very different from previous two families of matroids coming from algebra.

Some classes of matroids are defined by a certain property, instead of by a specific presentation (as it was in the previous examples). 

\begin{definition}\emph{(Base Orderable Matroid)}\label{DefinitionBaseOrderable} 
A matroid is \emph{base orderable} if for any two bases $B_1$ and $B_2$ there is a bijection $\pi:B_1\rightarrow B_2$ such that $B_1\cup\pi(e)\setminus e$ and $B_2\cup e\setminus\pi(e)$ are bases for any element $e\in B_1$.
\end{definition}

\begin{definition}\emph{(Strongly Base Orderable Matroid)}\label{DefinitionStronglyBaseOrderable} 
A matroid is \emph{stron\-gly base orderable} if for any two bases $B_1$ and $B_2$ there is a bijection $\pi:B_1\rightarrow B_2$ such that $B_1\cup\pi(A)\setminus A$ is a basis for any subset $A\subset B_1$.
\end{definition}

The condition of strongly base orderable matroid implies also that $B_2\cup A\setminus\pi(A)$ is a basis for any $A\subset B_1$. Moreover, we can assume that $\pi$ is the identity on $B_1\cap B_2$. The class of strongly base orderable matroids is closed under taking minors. As we will see in Section \ref{SectionWeakWhiteConjecture} of Chapter \ref{ChapterWhiteConjecture}, one of our results holds only for  strongly base orderable matroids.

\begin{definition}\emph{(Transversal Matroid)}\label{DefinitionTransversalMatroid} Let $E$ be a finite set, and let $\mathcal{A}=\{A_j:j\in J\}$ be a multiset of subsets of $E$. The following is a matroid
$$(E,\mathcal{I})=(E,\{A\subset E:\text{ there exists an injection }\phi:A\rightarrow J\text{ such that }$$ $$a\in A_{\phi(a)}\text{ for every }a\in A\}.$$
\end{definition}

In fact for any transversal matroid one can find a corresponding multiset $\mathcal{A}$ of cardinality equal to its rank. Another similar, and not so well-known class of matroids, is the class of laminar matroids, see \cite{Go09}.

\begin{definition}\emph{(Laminar Matroid)}\label{DefinitionLaminarMatroid} A family of sets is \emph{laminar} if any two sets are either disjoint or one is contained in the other. Let $\mathcal{F}$ be a laminar family of subsets of a finite set $E$, and let $k:\mathcal{F}\rightarrow\N$ be a function. The following is a matroid
$$(E,\mathcal{I})=(E,\{A\subset E:\left\vert A\cap F\right\vert\leq k(F)\text{ for every }F\in\mathcal{F}\}).$$ 
\end{definition}

As one can prove both transversal matroids and laminar matroids are strongly base orderable. Moreover, none of these classes is contained in the other.  

\begin{definition}\emph{(Regular Matroid)}\label{DefinitionRegularMatroid} 
A matroid is \emph{regular} if it is representable over every field.
\end{definition}

One can easily show that graphic matroids are regular. For characterizations of regular matroids see the paper of Tutte \cite{Tu58} (cf. \cite{Ox92}). Seymour \cite{Se80} proved a very interesting structural result. It asserts that every regular matroid may be constructed by piecing together graphic and cographic matroids and copies of a certain $10$-element matroid. 

Now we pass to operations on matroids. Firstly we are going to define two basic, restriction and contraction, which produce from a given matroid a matroid on a smaller ground set with useful properties. 

Let $M$ be a matroid on a ground set $E$ with the set of independent sets $\mathcal{I}$, the set of bases $\mathcal{B}$, rank function $r$, and let $F\subset E$. 

\begin{definition}\emph{(Restriction)} \emph{Restriction} of a matroid $M$ to the set $F$, denoted by $M\vert_{F}$ or $M\setminus (E\setminus F)$, is a matroid on the ground set $F$ satisfying:
\begin{enumerate}
\item $I\subset F$ is independent in $M\vert_{F}$ if $I$ is independent in $M$,
\item rank function of $M\vert_{F}$ equals to $r\vert_{F}$ -- restriction of rank function of $M$ to the set $F$.
\end{enumerate}
\end{definition}

\begin{definition}\emph{(Contraction)} \emph{Contraction} of the set $F$ in a matroid $M$, denoted by $M/F$, is a matroid on the ground set $E\setminus F$ satisfying:
\begin{enumerate}
\item $I\subset E\setminus F$ is independent in $M/F$ if for every $J\subset F$ independent in $M$ the set $I\cup J$ is independent in $M$,
\item rank function $r^{\prime}$ of $M/F$ satisfies $r^{\prime}(A)=r(A\cup F)-r(F)$ for every $A\subset E\setminus F$.
\end{enumerate}
\end{definition}

A matroid is called a \emph{minor} of another matroid if it is obtained from it by a sequence of restrictions and contractions. One can show that in fact any minor is obtained by at most two such operations. 

Matroid theory has a concept of duality. 

\begin{definition}\emph{(Dual Matroid)} \emph{Dual matroid} of a matroid $M$, denoted by $M^*$, is a matroid on the ground set $E$, such that:
\begin{enumerate}
\item $B\subset E$ is a basis in $M^*$ if $E\setminus B$ is a basis in $M$,
\item rank function $r^{\prime}$ of $M^*$ satisfies $r^{\prime}(A)=\left\vert A\right\vert-r(E)+r(E\setminus A)$ for every $A\subset E$.
\end{enumerate}
\end{definition}

Of course $(M^*)^*=M$. Now we can describe the duality between operations of restriction and contraction. Namely, we have 
$$M/F=(M^*\setminus F)^*\text{ and }M\setminus F=(M^*/F)^*.$$ 
Bases, circuits, etc. of the dual matroid are called \emph{cobases}, \emph{cocircuits}, etc. of the original matroid. 

It is also easy to define direct sum of matroids.

\begin{definition}\emph{(Direct Sum)}\label{DefinitionDirectSumOfMatroids}
Suppose that $M_1=(E_1,\mathcal{I}_1,r_1),\dots,M_k=(E_k,\mathcal{I}_k,r_k)$ are matroids on disjoint ground sets. Then, there exists a matroid $M_1\sqcup\dots\sqcup M_k=(E_1\cup\dots\cup E_k,\mathcal{I},r)$ satisfying:
\begin{enumerate}
\item $\mathcal{I}=\{I_1\cup\dots\cup I_k:I_1\in\mathcal{I}_1,\dots,I_k\in\mathcal{I}_k\}$,
\item $r(A)=r_1(A\cap E_1)+\dots+r_k(A\cap E_k)$ for every $A\subset E_1\cup\dots\cup E_k$.
\end{enumerate}
\end{definition}

As we will see in Theorem \ref{TheoremJoinOfMatroids} it is possible to extend this definition to matroids with not necessary disjoint ground sets.

One can also blow up an element of a matroid by adding `parallel' elements. This operation extends the idea of adding parallel edges in graphic matroids.

\begin{definition}\emph{(Blow Up)}\label{DefinitionBlowUp}
Let $M=(E,\mathcal{I},\mathcal{C})$ be a matroid and let $e\in E,f\notin E$. Then, there exists a matroid $M'=(E\cup f,\mathcal{I}',\mathcal{C}')$ satisfying:
\begin{enumerate}
\item $\mathcal{I}'=\mathcal{I}\cup\{I\cup f\setminus e:e\in I\in\mathcal{I}\}$,
\item $\mathcal{C}'=\mathcal{C}\cup\{C\cup f\setminus e:e\in C\in\mathcal{C}\}\cup\{e,f\}$.
\end{enumerate}
\end{definition}

The crucial fact about above definitions is that in each case the matroid with described set of independent sets, bases, circuits or rank function exists (for the proofs see \cite{Ox92}). 

\section{Matroid Union Theorem}

We will describe one of possible formulations of the Matroid Union Theorem, which is best for our applications.

\begin{theorem}[Nash-Williams, cf. \cite{Ox92}] \emph{(Matroid Union Theorem)}\label{TheoremMatroidUnion}
Let $M_1,$ $\dots,M_k$ be matroids on the same ground set $E$ with rank functions $r_1,\dots,r_k$ respectively. The following conditions are equivalent:
\begin{enumerate}
\item there exist sets $V_i$ with $V_1\cup\dots\cup V_k=E$, such that the set $V_i$ is independent in $M_i$ for each $i$,
\item for every $A\subset E$ we have $r_1(A)+\dots+r_k(A)\geq\left\vert A\right\vert$.
\end{enumerate}
\end{theorem}

The theorem follows from its weighted version, which we prove below. Let $w:E\rightarrow\N$ be a weight assignment of elements of $E$. A collection of (not necessarily different) subsets $V_1,\dots,V_d\subset E$ is said to be a $w$\emph{-covering} of $E$ if for every element $e\in E$ there are exactly $w(e)$ members of the collection containing $e$. If $w\equiv k$ for some integer $k$, then we write shortly $k$\emph{-covering}. By a covering we mean a $1$-covering.

\begin{theorem}[Nash-Williams, cf. \cite{Ox92}] \emph{(Weighted Matroid Union Theorem)}\label{TheoremWeightedMatroidUnion} Let $M_1,\dots,M_k$ be matroids on the same ground set $E$ with rank functions $r_1,\dots,r_k$ respectively. The following conditions are equivalent:
\begin{enumerate}
\item there exists a $w$-covering of $E$ by sets $V_1,\dots,V_k$ with $V_{i}$ independent in $M_i$ for each $i$,
\item for every $A\subset E$ we have $r_1(A)+\dots+r_k(A)\geq\sum_{e\in A}w(e)$.
\end{enumerate}
\end{theorem}

\begin{proof}
Implication $(1)\Rightarrow (2)$ is easy. For every $A\subset E$ we have 
$$r_1(A)+\dots+r_k(A)\geq\left\vert A\cap V_1\right\vert+\dots+\left\vert A\cap V_k\right\vert=\sum_{e\in A}w(e).$$ 

We focus on the opposite implication $(2)\Rightarrow (1)$, which we will prove by double induction. Firstly on the size of the ground set $E$, and secondly on the sum $\sum_{e\in E}w(e)$. 

The basis of induction, that is when $\left\vert E\right\vert=1$ or $w\equiv 0$, is clearly true. Consider a function $f:\mathcal{P}(E)\rightarrow\N$ defined by
\begin{equation*}
f(A)=r_1(A)+\dots+r_d(A)-\sum_{e\in A}w(e).
\end{equation*}
Notice that the condition $(2)$ guarantees that $f$ is nonnegative for all $A\subset E$. We distinguish two cases.

\begin{case1}
For some non-empty proper subset $A\subsetneq E$ there is $f(A)=0$. 
\end{case1}

Consider matroids $M_1,\dots,M_k$ restricted to $A$ and weight function $w\vert_A$. Since for them the condition $(2)$ of the theorem is satisfied, then from the inductive assumption there is a $w\vert_A$-covering by $U_1,\dots,U_k$ with $U_i$ is independent in $M_i\vert_A$. 

Consider also weight function $w\vert_{E\setminus A}$ and matroids $M_1^{\prime},\dots,M_d^{\prime}$ on the set $E\setminus A$ obtained by contracting set $A$ in $M_1,\dots,M_d$, that is $M_i^{\prime}=M_i/A$. Then $r_i^{\prime}(X)=r_i(X\cup A)-r_i(A)$, so they satisfy the condition $(2)$ of the theorem. From the inductive assumption there is a $w\vert_{E\setminus A}$-covering by $U_1^{\prime},\dots,U_k^{\prime}$ with $U_i^{\prime}$ is independent in $M_i/A$. Now unions $V_i=U_i\cup U_i^{\prime}$ form a $w$-covering of $E$, and $V_i$ is independent in $M_i$.

\begin{case1}
For every non-empty proper subset $A\subset E$ there is $f(A)>0$. 
\end{case1}

Pick an arbitrary element $e\in E$ with $w(e)>0$, and pick an arbitrary $i$ such that $r_i(e)>0$, which clearly exists because $f(e)>0$. Consider matroids $M_1,\dots,M_i^{\prime},\dots,M_k$, where $M_i^{\prime}=M_i/e$, and weight function $w'$ such that $w'\vert_{E\setminus e}\equiv w\vert_{E\setminus e}$ and $w'(e)=w(e)-1$. For them condition $(2)$ of the theorem still holds, so from the inductive assumption there is a $w'$-covering by sets $V_1,\dots,V_i',\dots,V_k$ independent in corresponding matroids. Observe that if we define $V_i=V_i'\cup e$, then $V_1,\dots,V_i,\dots,V_k$ will be an eligible $w$-covering for matroids $M_1,\dots,M_i,\dots,M_k$.
\end{proof}

Another formulation of the Matroid Union Theorem describes the rank function of the join of matroids. We derive it as a corollary of the previous formulation.

\begin{theorem}[Nash-Williams, cf. \cite{Ox92}] \emph{(Join of Matroids)}\label{TheoremJoinOfMatroids}
Let $M_1=(E,\mathcal{I}_1,r_1),$ $\dots,M_k=(E,\mathcal{I}_k,r_k)$ be matroids on the same ground set $E$. Then, there exists a matroid $M_1\vee\dots\vee M_k=(E,\mathcal{I},r)$ satisfying:
\begin{enumerate}
\item $\mathcal{I}=\{I_1\cup\dots\cup I_k:I_1\in\mathcal{I}_1,\dots,I_k\in\mathcal{I}_k\}$,
\item $r(A)=\min\{r_1(B)+\dots+r_k(B)+\left\vert A\setminus B\right\vert:B\subset A\}$.
\end{enumerate}
\end{theorem}

\begin{proof} 
We present a proof in the case $k=2$, a generalization to arbitrary number of matroids is straightforward.

According to Definition \ref{DefinitionRank}, matroid $M_1\vee M_2=(E,r)$ exists if conditions $(1)-(3)$ of rank function hold. It is easy to check that first two are satisfied. To check the submodularity axiom $r(A\cup B)+r(A\cap B)\leq r(A)+r(B)$ assume that minima from the definition of $r(A)$ and $r(B)$ are realized for sets $C,D$ respectively. We have
$$r(A\cup B)+r(A\cap B)\leq r_1(C\cup D)+r_2(C\cup D)+\left\vert A\cup B\setminus(C\cup D)\right\vert+$$ $$r_1(C\cap D)+r_2(C\cap D)+\left\vert A\cap B\setminus(C\cap D)\right\vert$$ $$\leq r_1(C)+r_2(C)+\left\vert A\setminus C\right\vert+r_1(D)+r_2(D)+\left\vert B\setminus D\right\vert=r(A)+r(B),$$
where the second inequality follows from submodularity of $r_1$ and $r_2$. 

Suppose now that $r(A)=\left\vert A\right\vert$. It means that for each $B\subset A$ there is $$r_1(B)+r_2(B)+\left\vert A\setminus B\right\vert\geq\left\vert A\right\vert\text{, so }r_1(B)+r_2(B)\geq\left\vert B\right\vert.$$
From Matroid Union Theorem \ref{TheoremMatroidUnion} it follows that $A=I_1\cup I_2$ for some $I_1\in\mathcal{I}_1,I_2\in\mathcal{I}_2$. Contrary, if $A=I_1\cup I_2$ is the sum of two independent sets, then for each $B\subset A$ we have
$$r_1(B)+r_2(B)+\left\vert A\setminus B\right\vert\geq r_1(B\cap I_1)+r_2(B\cap I_2)+\left\vert A\setminus B\right\vert\geq$$
$$\left\vert B\cap I_1\right\vert+\left\vert B\cap I_2\right\vert+\left\vert A\setminus B\right\vert\geq\left\vert A\right\vert,$$
so $A$ is independent.
\end{proof}

Notice that a special case of join of matroids is the direct sum of matroids (see Definition \ref{DefinitionDirectSumOfMatroids}). Direct sum of matroids on disjoint ground sets is just the join of their extensions to the sum of all ground sets.

\begin{corollary}
A matroid $M=(E,r)$ can be covered with $k$ independent sets if and only if for every $A\subset E$ there is $kr(A)\geq\left\vert A\right\vert$.
\end{corollary}

\begin{proof}
It is exactly Matroid Union Theorem \ref{TheoremMatroidUnion} applied to $M_1=\dots=M_k=M$. We can also easily get it from Theorem \ref{TheoremJoinOfMatroids}. Namely, matroid $M$ can be covered with $k$ independent sets if and only if $E$ is independent in set (rank of $E$ is $\left\vert E\right\vert$) in the matroid $M'=M_1\vee\dots\vee M_k$, where all $M_i=M$. Rank of $E$ in $M'$ equals to $\min\{kr(A)+\left\vert E\setminus A\right\vert:A\subset E\}$. 
\end{proof}

\begin{corollary}
A matroid $M=(E,r)$ has $k$ disjoint bases if and only if for every $A\subset E$ there is $kr(A)+\left\vert E\setminus A\right\vert\geq kr(E)$.
\end{corollary}

\begin{proof}
Matroid $M$ has $k$ disjoint bases if and only if $E$ has rank at least $kr(E)$ in the matroid $M'=M_1\vee\dots\vee M_k$, where all $M_i=M$. Due to Theorem \ref{TheoremJoinOfMatroids} rank of $E$ in $M'$ equals to $\min\{kr(A)+\left\vert E\setminus A\right\vert:A\subset E\}$. 
\end{proof}

The last theorem of this Section gives an explicit formula, in terms of rank functions, for the largest common independent set in two matroids. Unfortunately, it does not have an easy generalization to larger number of matroids.

\begin{theorem}[Edmonds, cf. \cite{Ox92}] \emph{(Intersection of Two Matroids)}\label{TheoremEdmonds}
Let $M_1=(E,\mathcal{I}_1,r_1)$ and $M_2=(E,\mathcal{I}_2,r_2)$ be two matroids on the same ground set $E$. Then $$\max\{\left\vert I\right\vert: I\in\mathcal{I}_1\cap\mathcal{I}_2\}=\min\{r_1(A)+r_2(E\setminus A):A\subset E\}.$$
\end{theorem}

\begin{proof} 
Inequality $\leq$ is obvious. To prove $\geq$ suppose that for each $A\subset E$ there is $r_1(A)+r_2(E\setminus A)\geq k$. Observe that $r_{M_1\vee M_2^*}(E)\geq k+r_{M_2^*}(E)$, because for each $A\subset E$ there is $$r_1(A)+\left\vert A\right\vert-r_2(E)+r_2(E\setminus A)+\left\vert E\setminus A\right\vert\geq k+\left\vert E\right\vert-r_2(E)=k+r_{M_2^*}(E).$$ Thus there are two disjoint sets $I_1$ independent in $M_1$ and $I_2$ independent in $M_2^*$, such that $\left\vert I_1\right\vert+\left\vert I_2\right\vert\geq k+r_{M_2^*}(E)$. Suppose $c=r_{M_2^*}(E)-\left\vert I_2\right\vert$. Extend $I_2$ with $c$ elements to a basis $B^*$ of $M_2^*$. Now $I_1\cap (E\setminus B^*)$ is independent in both $M_1$ and $M_2$. Moreover, it has at least $\left\vert I_1\right\vert-c\geq k$ elements.
\end{proof}

\part{Coloring Games on Matroids}

\chapter{Colorings of a Matroid}\label{ChapterColoringsOfMatroid}

This Chapter extends the concept of a proper graph coloring to matroids. Namely, a coloring of a matroid is proper if elements of the same color form an independent set. In that way we get a notion of chromatic number and list chromatic number of a matroid (several game versions of this parameters will be studied in the following Chapter \ref{ChapterGameColoringsMatroid}). We present explicit formula of Edmonds \cite{Ed65} for chromatic number in terms of rank function, and a theorem of Seymour \cite{Se98} asserting that the list chromatic number of a matroid equals to the chromatic number. 

The main result of this Chapter, Theorem \ref{TheoremGeneralizationOfSeymour}, generalizes Seymour's theorem from lists of constant size to the case when lists have arbitrary fixed size. We prove that a matroid is colorable from any lists of a fixed size if and only if it is colorable from particular lists of that size (this theorem will be later even further generalized, see Theorem \ref{TheoremOnLineGeneralizationOfSeymour}). 

The last Section shows applications of Theorem \ref{TheoremGeneralizationOfSeymour} to several basis exchange properties. The results of this Chapter come from our paper \cite{La14}.

\section{Chromatic Number}

Let $M$ be a matroid on a ground set $E$. A \emph{coloring} of $M$ is an assignment of colors (usually natural numbers) to elements of $E$. A coloring is \emph{proper} if elements of the same color form an independent set. For the rest of this thesis whenever we write the term coloring, we always mean proper coloring. 

An element $e\in E$ in a matroid $M=(E,r)$ is called a \emph{loop} if $r(e)=0$ (in analogy to a loop in a graph, which is a loop in the corresponding graphic matroid). Notice that if a matroid contains a loop, then it does not admit any proper coloring. On the other hand, if a matroid does not have any loop, it is \emph{loopless}, then it has at least one proper coloring. Thus, when considering colorings we restrict to loopless matroids.

\begin{definition}\emph{(Chromatic Number)}\label{DefinitionChromaticNumber} The \emph{chromatic number} of a loopless matroid $M$, denoted by $\chi(M)$, is the minimum number of colors in a proper coloring of $M$.
\end{definition}

For instance, if $M=M(G)$ is a graphic matroid, then $\chi(M)$ is the least number of colors needed to color edges of $G$ so that no cycle is monochromatic. This number is known as the \emph{arboricity} of a graph $G$ and should not be confused with the usual chromatic number $\chi(G)$ of a graph $G$.

In further analogy to graph theory we define fractional chromatic number. For a fixed integer $a$ an $a$\emph{-coloring} of a matroid $M$ on a ground set $E$ is an assignment of $a$ colors to each element of $E$. Still a coloring is \emph{proper} if for each color elements to which it is assigned form an independent set in $M$. Notice, that there is no difference between proper $a$-coloring of a matroid and $a$-covering with independent sets. 

\begin{definition}\emph{(Fractional Chromatic Number)}\label{DefinitionFractionalChromaticNumber}
The \emph{fractional chromatic number} of a loopless matroid $M$, denoted by $\chi_f(M)$, is the infimum of fractions $\frac{b}{a}$ such that $M$ has a proper $a$-coloring with $b$ colors. 
\end{definition}

For a matroid the chromatic number as well as the fractional chromatic number can be easily expressed in terms of a rank function. Extending a theorem of Nash-Williams \cite{Na64} for graph arboricity, Edmonds \cite{Ed65} proved the following explicit formula for the chromatic number of a matroid. It was observed that his proof gives also a formula for fractional chromatic number. 

\begin{theorem}[Edmonds, \cite{Ed65}] \emph{(Formula for Chromatic Number)}\label{TheoremFractionalChromaticNumber} For every looples matroid $M=(E,r)$ we have \begin{equation*}
\chi_f(M)=\max_{\emptyset\neq A\subset E}\frac{\left\vert A\right\vert}{r(A)}\text{ and } \chi(M)=\max_{\emptyset\neq A\subset E}\left\lceil\frac{\left\vert A\right\vert}{r(A)}\right\rceil.
\end{equation*}
\end{theorem}

\begin{proof} 
Denote $$\Delta(M)=\max_{\emptyset\neq A\subset E}\frac{\left\vert A\right\vert}{r(A)}.$$ 

Suppose $M$ has an $a$-covering with $b$ independent sets $V_1,\dots,V_b$, then for each $\emptyset\neq A\subset E$ we have $$br(A)\geq r(A\cap V_1)+\dots+r(A\cap V_b)=\left\vert A\cap V_1\right\vert+\dots+\left\vert A\cap V_b\right\vert=a\left\vert A\right\vert.$$ Hence $\frac{b}{a}\geq\frac{\left\vert A\right\vert}{r(A)}$, and as a consequence $\chi_f(M)\geq\Delta$. Since $\chi(M)$ is an integer, and obviously $\chi(M)\geq\chi_f(M)$, we get immediately that $\chi(M)\geq\left\lceil\Delta\right\rceil$. 

To prove opposite inequalities suppose that the maximum in the definition of $\Delta$ is reached for some non-empty subset $A\subset E$, that is $\Delta(M)=\frac{\left\vert A\right\vert}{r(A)}$. 

For fractional chromatic number consider $\left\vert A\right\vert$ matroids $M_1,\dots,M_{\left\vert A\right\vert}$ on the ground set $E$, all equal to $M$. Apply Weighted Matroid Union Theorem \ref{TheoremWeightedMatroidUnion} with weight function $w\equiv r(A)$. For each $\emptyset\neq B\subset E$ we have 
$$r_1(B)+\dots+r_{\left\vert A\right\vert}(B)=\left\vert A\right\vert r(B)\geq\left\vert B\right\vert r(A),$$ 
so the second condition of the theorem is satisfied. Hence, there exists an $r(A)$-covering of $M$ with $\left\vert A\right\vert$ independent sets. This means that $\chi_f(M)\leq\Delta$. 

For chromatic number consider $\left\lceil\Delta\right\rceil$ matroids $M_1,\dots,M_{\left\lceil\Delta\right\rceil}$ on the ground set $E$, all equal to $M$. We will apply Matroid Union Theorem \ref{TheoremMatroidUnion}. For each $\emptyset\neq B\subset E$ we have $$r_1(B)+\dots+r_{\left\lceil\Delta\right\rceil}(B)=\left\lceil\Delta\right\rceil r(B)=\left\lceil\frac{\left\vert A\right\vert}{r(A)}\right\rceil r(B)\geq\frac{\left\vert A\right\vert}{r(A)}r(B)\geq\left\vert B\right\vert.$$ 
The second condition of the theorem holds, hence there exists a proper coloring of $M$ with $\left\lceil\Delta\right\rceil$ independent sets. This means that $\chi_f(M)\leq\left\lceil\Delta\right\rceil$. 
\end{proof}

In particular, the chromatic number of a matroid is less than one bigger than its fractional chromatic number. In contrary for graphs chromatic number is not bounded from above by any function of fractional chromatic number (see for example Kneser's graphs \cite{Gr02}). 

Obviously in graph theory there is no chance for an analogical formula, however there is one example of an almost exact formula for the chromatic number. Vizing \cite{Vi64} proved that the chromatic index (chromatic number of line graph -- incidence graph of edges of a graph) is always equal to the maximum degree or to the maximum degree plus one.

\section{List Chromatic Number}\label{SectionListChromaticNumber}

Suppose each element $e$ of the ground set $E$ of a matroid $M$ is assigned with a set (a \emph{list}) of colors $L(e)$. In analogy to the list coloring of graphs (initiated by Vizing \cite{Vi76}, and independently by Erd\H{o}s, Rubin and Taylor \cite{ErRuTa80}), we want to find a proper coloring of $M$ with additional restriction that each element gets a color from its list. By a \emph{list assignment} (or simply \emph{lists}) we mean a function $L:E\rightarrow\mathcal{P}(\N)$, and by the \emph{size} of $L$ we mean the function $\ell:E\rightarrow\N$ with $\left\vert L(e)\right\vert=\ell(e)$ for each $e\in E$. A matroid $M$ is said to be \emph{colorable from lists} $L$ if there exists a proper coloring $c:E\rightarrow\N$ of $M$, such that $c(e)\in L(e)$ for every $e\in E$. 

\begin{definition}\emph{(List Chromatic Number)}\label{DefinitionListChromaticNumber} The \emph{list chromatic number} of a loopless matroid $M$, denoted by $\ch(M)$, is the minimum number $k$ such that $M$ is colorable from any lists of size at least $k$.
\end{definition}

Clearly, we have $\ch(M)\geq \chi (M)$. The theorem of Seymour \cite{Se98} asserts that actually we have equality here. We provide a proof of this result for completeness. It is a simple application of Matroid Union Theorem \ref{TheoremMatroidUnion}.

\begin{theorem}[Seymour, \cite{Se98}]\label{TheoremSeymour}
For every loopless matroid $M$ we have equality $\ch(M)=\chi(M)$.
\end{theorem}

\begin{proof}
Suppose $\chi(M)=k$, and take any list assignment $L:E\rightarrow\mathcal{P}(\{1,\dots,d\})$ of size at least $k$. Let $Q_i=\{e\in E:i\in L(e)\}$, and let $M_i=(E,r_i)$ be the restriction of $M$ to $Q_i$ with the ground set extended to $E$. From the assumption there exists a $1$-covering with $k$ independent sets $V_1,\dots,V_k$. Then, for every $A\subset E$ we have $$\sum_{1\leq j\leq k}\frac{1}{k}\sum_{1\leq i\leq d}\left\vert V_j\cap A\cap Q_i\right\vert\geq\sum_{1\leq j\leq k}\left\vert V_j\cap A\right\vert=\left\vert A\right\vert,$$
since each element $e\in V_j\cap A$ belongs to $Q_i$ for at least $k$ values of $i$. But $r_i(A)\geq\left\vert V_j\cap A\cap Q_i\right\vert$, so $$\sum_{1\leq i\leq d}r_i(A)=\sum_{1\leq j\leq k}\frac{1}{k}\sum_{1\leq i\leq d}r_i(A)\geq\left\vert A\right\vert.$$ The result follows from Matroid Union Theorem \ref{TheoremMatroidUnion}.
\end{proof}

For a weight assignment $w$ and a list assignment $L$ we say that $M$ is $w$\emph{-colorable from lists} $L$, if it is possible to choose for each $e\in E$ a set of $w(e)$ colors from its list $L(e)$, such that for each color elements to which it is assigned form an independent set in $M$. In other words each list $L(e)$ contains a subset $L^{\prime }(e)$ of size $w(e)$ such that choosing any color from $L^{\prime}(e)$ for $e$ results in a proper coloring of a matroid $M$. 

\begin{definition}\emph{(Fractional List Chromatic Number)}\label{DefinitionFractionalListChromaticNumber}
The \emph{fractional list chromatic number} of a loopless matroid $M$, denoted by $\ch_f(M)$, is the infimum of fractions $\frac{b}{a}$ for which $M$ is $a$-colorable from any lists of size at least $b$.
\end{definition}

For graphs list chromatic number is not bounded from above by any function of chromatic number, as it can already be arbitrarily large for graphs with chromatic number two. Dinitz stated a conjecture that chromatic number and list chromatic number coincide for line graphs. It was confirmed by Galvin \cite{Ga95} for bipartite graphs, and for arbitrary graphs in the asymptotic sense by Kahn \cite{Ka00}. Surprisingly, Alon, Tuza and Voigt \cite{AlTuVo97} proved that fractional versions of chromatic number and list chromatic number coincide for all graphs. 

We generalize Seymour's Theorem \ref{TheoremSeymour} to the setting, where sizes of lists are still fixed, but not necessarily equal. 

\begin{theorem}[Lasoń, \cite{La14}]\emph{(Generalization of Seymour's Theorem)}\label{TheoremGeneralizationOfSeymour} Let $M=(E,r)$ be a matroid, and let $\ell:E\rightarrow\N$, $w:E\rightarrow\N$ be list size and weight functions. Then the following conditions are equivalent:
\begin{enumerate}
\item for each $A\subset E$ there is $\sum_{i\in \N_{+}}r(\{e\in A:\ell(e)\geq i\})\geq \sum_{e\in A}w(e)$,
\item matroid $M$ is $w$-colorable from lists $L_{\ell}(e)=\{1,\dots,\ell(e)\}$,
\item matroid $M$ is $w$-colorable from any lists of size $\ell$.
\end{enumerate}
\end{theorem}

\begin{proof}
Let $L$ be a fixed list assignment of size $\ell$, and let us denote by $Q_i=\{e\in E:i\in L(e)\}$ and $d=\max\bigcup_{e\in E}L(e)$ (and respectively $Q^{\ell}_{i}$ for lists $L_{\ell}$). Consider matroids $M_1,\dots,M_d$, such that $M_i$ equals to $M\vert_{Q_i}$ with the ground set extended to $E$. Observe that matroid $M$ is properly $w$-colorable from lists $L$ if and only if its
ground set $E$ can be $w$-covered by sets $V_1,\dots,V_d$, such that $V_i$
is independent in $M_i$. Namely two expressions to cover and to color have the same meaning. From Weighted Matroid Union Theorem \ref{TheoremWeightedMatroidUnion} we get that this condition is equivalent to the fact, that for each $A\subset E$ there is an inequality 
$$r(A\cap Q_1)+\cdots+r(A\cap Q_d)\geq\sum_{e\in A}w(e).$$

Now equivalence between conditions $(1)$ and $(2)$ follows from the fact
that $\{e\in A:\ell(e)\geq i\}=A\cap Q^{\ell}_{i}$.

Implication $(3)\Rightarrow(2)$ is obvious. To prove $(2)\Rightarrow(3)$ it
is enough to show that for each subset $A\subset E$ there is an inequality 
$$r(A\cap Q_1)+\cdots+r(A\cap Q_d)\geq r(A\cap Q^{\ell}_1)+\cdots+r(A\cap Q^{\ell}_d).$$
We will use the notion $\sum$ or $+$ for the sum of multisets. Since both list assignments $L_{\ell}$ and $L$ are of size $\ell$, we get that $\sum_iQ_i=\sum_i Q^{\ell}_{i}=\sum_{e\in E}\ell(e)\{e\}$
as multisets. Observe that to prove the inequality, whenever two sets $Q_i$ and $Q_j$ are incomparable in the inclusion order, we may replace them by sets 
$Q_i\cup Q_j$, and $Q_i\cap Q_j$. It is because:

\begin{enumerate}
\item $r(A\cap Q_i)+r(A\cap Q_j)\geq r(A\cap(Q_i\cup Q_j))+r(A\cap(Q_i\cap
Q_j))$, by submodularity of $r$ (see Definition \ref{DefinitionRank}),
\item $Q_i+Q_j=(Q_i\cup Q_j)+(Q_i\cap Q_j)$ as multisets,
\item $\left\vert Q_i \right\vert^2 +\left\vert Q_j \right\vert^2 < \left\vert Q_i\cup Q_j\right\vert^2
+\left\vert Q_i\cap Q_j\right\vert^2$.
\end{enumerate}

The last parameter grows, and it is bounded (because the number of non-empty
sets, and their sizes are bounded). Hence, after a finite number of steps the
replacement procedure stops. Then sets $Q_i$ are linearly ordered by inclusion.
Let us reorder them in such a way that $Q_1\supset Q_2\supset\dots\supset Q_d$. Then, it is
easy to see that $Q_1=Q^{\ell}_1,Q_2=Q^{\ell}_2,\dots,Q_d=Q^{\ell}_d$, so the inequality $r(A\cap Q_1)+\cdots+r(A\cap Q_d)\geq r(A\cap Q^{\ell}_1)+\cdots+r(A\cap Q^{\ell}_d)$ is satisfied. This proves the assertion.
\end{proof}

Seymour's Theorem \ref{TheoremSeymour} follows from the above result by taking a constant
size function $\ell\equiv\chi(M)$ and a weight function $w\equiv 1$. We get also that $\ch_f(M)=\chi_f(M)$. Our theorem implies as well the following stronger statement.

\begin{corollary}
Let $I_{1},\dots ,I_{k}$ be a partition of the ground set of a matroid $M$ into independent sets, and let $\ell(e)=i$ whenever $e\in I_{i}$. Then matroid $M$ is colorable from any lists of size $\ell$.
\end{corollary}

In the next Section we show applications of our Generalization of Seymour's Theorem \ref{TheoremGeneralizationOfSeymour} to basis exchange properties.

\section{Basis Exchange Properties}\label{SectionBasisExchangeProperties}

Let $E$ be a finite set. For a family $\mathcal{B}\subset\mathcal{P}(E)$ of subsets of $E$ consider the following conditions:
\begin{enumerate}
\item $\mathcal{B}$ is non-empty,
\item for every $B_1,B_2\in\mathcal{B}$ and $e\in B_1\setminus B_2$ there exists $f\in B_2\setminus B_1$, such that $B_1\cup f\setminus e\in\mathcal{B}$,
\item for every $B_1,B_2\in\mathcal{B}$ and $e\in B_1\setminus B_2$ there exists $f\in B_2\setminus B_1$, such that $B_1\cup f\setminus e\in\mathcal{B}$ and $B_2\cup e\setminus f\in\mathcal{B}$,
\item for every $B_1,B_2\in\mathcal{B}$ and $A_1\subset B_1$ there exists $A_2\subset B_2$, such that $B_1\cup A_2\setminus A_1\in\mathcal{B}$ and $B_2\cup A_1\setminus A_2\in\mathcal{B}$.
\end{enumerate}

According to Definition \ref{DefinitionBases}, family $\mathcal{B}\subset\mathcal{P}(E)$ is a set of bases of a matroid if conditions $(1)$ and $(2)$ are satisfied. It is a nice exercise to show that in this case condition $(3)$ also holds. It is called \emph{symmetric exchange property}, and was discovered by Brualdi \cite{Br69}. We will often refer to it in Chapter \ref{ChapterWhiteConjecture}. 

Surprisingly even condition $(4)$, known as \emph{multiple symmetric exchange property}, is true (for simple proofs see \cite{LaLu12,Wo74}, and for more exchange properties \cite{La14,Ku86}). 

We will demonstrate usefulness of our Generalization of Seymour's Theorem \ref{TheoremGeneralizationOfSeymour} by giving easy proofs of several basis exchange properties, in particular of multiple symmetric exchange property. The idea of the proofs is to choose a suitable list assignment which guarantees existence of a required coloring. The crucial point in our Generalization of Seymour's Theorem is that lists may have distinct sizes. In particular, if a list of an element has size one, then the color of that element is already determined by a list assignment.

\begin{proposition}\emph{(Multiple Symmetric Exchange Property)}\label{PropositionMultipleSymmetricExchange} 
Let $B_1$ and $B_2$ be bases of a matroid $M$. Then for every $A_1\subset B_1$ there exists $A_2\subset B_2$, such that $(B_1\setminus A_1)\cup A_2$ and $(B_2\setminus A_2)\cup A_1$ are both bases.
\end{proposition}

\begin{proof} 
Observe that we can restrict our attention to the case when $B_1$ and $B_2$ are disjoint. Indeed, otherwise consider matroid $M/(B_1\cap B_2)$, its disjoint bases $B_1\setminus B_2,B_2\setminus B_1$ and a set $A_1\setminus B_2\subset B_1\setminus B_2$. If for them there exists an appropriate $A_2$, then $A_2\cup (A_1\cap B_2)$ solves the original problem.

When bases $B_1,B_2$ are disjoint, then restrict matroid $M$ to their union. Let $L$ assign list $\{1\}$ to elements of $A_1$, list $\{2\}$ to elements of $B_1\setminus A_1$ and list $\{1,2\}$ to elements of $B_2$. Observe that for $w\equiv 1$ the condition $(1)$ of Theorem \ref{TheoremGeneralizationOfSeymour} is satisfied, so we get that there is a $1$-coloring from lists $L$. Denote by $C_1$ elements colored with $1$, and by $C_2$ those colored with $2$. Now $A_2=C_2\cap B_2$ is a good choice, since sets $(B_1\setminus A_1)\cup A_2=C_2$ and $(B_2\setminus A_2)\cup A_1=C_1$ are independent. 
\end{proof}

We formulate also a more general multiple symmetric exchange property for independent sets. We will use it in Section \ref{SectionOnLineListColoring} of Chapter \ref{ChapterGameColoringsMatroid}. Exactly the same proof applies.

\begin{proposition}\label{PropositionMultipleIndependentSetExchange} 
Let $I_1$ and $I_2$ be independent sets in a matroid $M$. Then for every $A_1\subset I_1$ there exists $A_2\subset I_2$, such that $(I_1\setminus A_1)\cup A_2$ and $(I_2\setminus A_2)\cup A_1$ are both independent sets.
\end{proposition}

Multiple Symmetric Exchange Property \ref{PropositionMultipleSymmetricExchange} can be slightly generalized. Instead of having a partition of one of bases into two pieces we can have an arbitrary partition of it. We prove that for any such partition there exists a partition of the second basis which is consistent in two different ways.

\begin{proposition} 
Let $A$ and $B$ be bases of a matroid $M$. Then for every partition $B_1\sqcup\dots\sqcup B_k=B$ there exists a partition $A_1\sqcup\dots\sqcup A_k=A$, such that $(B\setminus B_i)\cup A_i$ are all bases for $1\leq i\leq k$.
\end{proposition}

\begin{proof}
Analogously to the proof of Proposition \ref{PropositionMultipleSymmetricExchange} we can assume that bases $A,B$ are disjoint (on their intersection sets $A_i$ and $B_i$ should coincide). 

When bases $A$ and $B$ are disjoint, then restrict matroid $M$ to their sum. Let $L$ assign list $\{1,\dots,k\}\setminus\{i\}$ to elements of $B_i$, and list $\{1,\dots,k\}$ to elements of $A$. Define weight function by $w\equiv k-1$ on $B$ and $w\equiv 1$ on $A$. Observe that the condition $(1)$ of Theorem \ref{TheoremGeneralizationOfSeymour} is satisfied, so there exists a $w$-coloring from lists $L$. Denote by $C_i$ elements colored with $i$. Now $A_i=C_i\cap A$ is a good partition, since sets $(B\setminus B_i)\cup A_i=C_i$ are independent. 
\end{proof}

\begin{proposition} 
Let $A$ and $B$ be bases of a matroid $M$. Then for every partition $B_1\sqcup\dots\sqcup B_k=B$ there exists a partition $A_1\sqcup\dots\sqcup A_k=A$, such that $(A\setminus A_i)\cup B_i$ are all bases for $1\leq i\leq k$.
\end{proposition}

\begin{proof}
Analogously to the proofs of previous propositions we can assume that bases $A$ and $B$ are disjoint. Then we restrict matroid $M$ to $A\cup B$. Let $L$ assign list $\{i\}$ to elements of $B_i$, and list $\{1,\dots,k\}$ to elements of $A$. Define weight function by $w\equiv 1$ on $B$ and $w\equiv k-1$ on $A$. Observe that the condition $(1)$ of Theorem \ref{TheoremGeneralizationOfSeymour} holds, so there is a $w$-coloring from lists $L$. Denote by $C_i$ elements colored with $i$. Now $A_i=A\setminus C_i$ is a good partition, since sets $(A\setminus A_i)\cup B_i=C_i$ are independent. 
\end{proof}

\chapter{Game Colorings of a Matroid}\label{ChapterGameColoringsMatroid}

The typical problem in game coloring is to get optimal coloring of a given combinatorial structure (graph, poset, matroid, etc.) even if one of the players craftily aims to fail this task.

The scheme of our games is the following. There are two players -- Alice and Bob. Both know the structure of a fixed loopless matroid $M$. During the play, they (or one of them) are coloring properly elements of the ground set of $M$ obeying some rules. The game ends when the whole matroid has been colored, or if they arrive to a partial coloring that cannot be further extended by making a legal move. Alice wins in the first case, so her goal is to color the whole matroid. Therefore she is called a `good player'. Bob's goal is the opposite, and he is called a `bad guy'. 

The chromatic parameter of this kind of game, denoted by $\chi_{game}(M)$, is the minimum number of colors available for an element (usually the size of the set of colors) for which Alice has a winning strategy. We always have 
$$\chi_{game}(M)\geq\chi(M),$$ 
since Alice's win results in a proper coloring of $M$. Our goal is to bound $\chi_{game}$ form above by the best possible function of $\chi$. That is, to find the pointwise smallest function $f:\N\rightarrow\N$, such that every loopless matroid $M$ satisfies
$$\chi_{game}(M)\leq f(\chi(M)).$$

We consider three coloring games on matroids, all of which were initially investigated for graphs (for graphs function $f$ usually does not exist). In two cases we get optimal functions, and in the third case an almost optimal. 

In Section \ref{SectionMatroidGameColoring} we study game coloring, introduced for planar graphs by Brams (cf. \cite{Ga81}), and independently by Bodlaender \cite{Bo91}. In Theorem \ref{TheoremGeneralizedGame} we prove that every matroid $M$ satisfies $\chi_{g}(M)\leq 2\chi (M)$. This improves and extends a theorem of Bartnicki, Grytczuk, and Kierstead \cite{BaGrKi08}, who proved that $\chi_{g}(M)\leq 3\chi(M)$ holds for every graphic matroid $M$. Moreover, in Theorem \ref{TheoremLowerBoundForGameColoring} we show that our bound is in general almost tight (up to at most one).

In Section \ref{SectionOnLineListColoring} we analyse on-line list coloring, introduced for graphs by Schauz \cite{Sc09} (see also \cite{Zh09}). In Theorem \ref{TheoremOnLineGeneralizationOfSeymour} we prove that the corresponding game parameter $\ch_{on-line}(M)$ in fact equals to $\chi(M)$. This generalizes to the on-line setting our Generalization of Seymour's Theorem \ref{TheoremGeneralizationOfSeymour} (and in particular Seymour's Theorem \ref{TheoremSeymour}).

In the last Section \ref{SectionIndicatedColoring} we examine indicated coloring, introduced for graphs by Grytczuk (cf. \cite{Gr12}). Theorem \ref{TheoremGeneralIndicatedColoring} asserts that $\chi_i(M)=\chi(M)$.

We hope that our results add a new aspect of structural type, by showing that most of pathological phenomena appearing for graph coloring games are no longer present in the realm of matroids. The results of consecutive Sections come from our papers \cite{La12}, \cite{LaLu12} (with Wojciech Lubawski), and \cite{La13b} respectively.

\section{Game Coloring}\label{SectionMatroidGameColoring}

In this Section we study a game-theoretic variant of the chromatic number of a matroid defined as follows. 
Two players, Alice and Bob, alternately color elements of the ground set $E$ of a matroid $M$ using a fixed set of colors $C$. The only rule that both players have to obey is that at any moment of the play, all elements in the same color must form an independent set (it is a proper coloring). The game ends when the whole matroid has been colored, or if they arrive to a partial coloring that cannot be further extended (what happens when trying to color any uncolored element with any possible color results in a monochromatic circuit of $M$). Alice wins in the first case, while Bob in the second. The \emph{game chromatic number} of a matroid $M$, denoted by $\chi_{g}(M)$, is the minimum size of the set of colors $C$ for which Alice has a winning strategy. 

The above game is a matroidal analog of a well studied graph game coloring, which was introduced by Brams (cf. \cite{Ga81}) for planar graphs with a motivation to give an easier proof of the Four Color Theorem. The game was independently reinvented by Bodlaender \cite{Bo91}, and since then the topic developed into several directions leading to interesting results, sophisticated methods and challenging open problems (see a recent survey \cite{BaGrKiZh07}). 

The first step in studying game chromatic number of matroids was made by Bartnicki, Grytczuk, and Kierstead \cite{BaGrKi08}, who proved that for every graphic matroid $M$ inequality $\chi_{g}(M)\leq 3\chi(M)$ holds. In Theorem \ref{TheoremGeneralizedGame} we improve and extend this result by showing that for every loopless matroid $M$ we have $\chi_{g}(M)\leq 2\chi (M)$. This gives a nearly tight bound, since in Theorem \ref{TheoremLowerBoundForGameColoring} we provide a class of matroids satisfying $\chi(M_k)=k$ and $\chi_{g}(M_k)=2k-1$ for every $k\geq 3$. Our bounds remain true also for the fractional parameters, as well as for list version of the game chromatic number. The results of this Section come from our paper \cite{La12}.

We pass to the proof of an upper bound. To achieve our goal we shall need a more general version of the matroid coloring game. Let $M_{1},\dots ,M_{d}$ be matroids on the same ground set $E$. The set of colors is restricted to $\{1,\dots,d\}$. As before, the players alternately color elements of $E$, but now for all $i$ the set of elements colored with $i$ must be independent in the matroid $M_{i}$. We call this game a \emph{coloring game on} $M_1,\dots,M_d$. The initial game on $M$ with $d$ colors coincides with the coloring game on $M_{1}=\dots=M_{d}=M$.

\begin{theorem}[Lasoń, \cite{La12}]\label{TheoremGeneralizedGame}
Let $M_{1},\dots ,M_{d}$ be matroids on the ground set $E$. If $E$ has a $2$-covering by sets $V_{1},\dots,V_{d}$, with $V_{i}$ independent in $M_{i}$, then Alice has a winning strategy in the coloring game on $M_{1},\dots ,M_{d}$. In particular, for every loopless matroid $M$ we have $\chi_{g}(M)\leq 2\chi (M)$. 
\end{theorem}

\begin{proof}
Let us fix a $2$-covering of $E$ by sets $V_1,\dots,V_d$ independent in corresponding matroids. Denote by $C=U_{1}\cup\dots\cup U_{d}$ the set of colored elements (at some moment of the play), where $U_{i}$ is the set of elements colored with $i$. Alice will be coloring element $e$ which belongs to $V_{i}$ and $V_j$ always with color $i$ or $j$, and she will try to keep the following invariant after each of her move:
\begin{equation}
\text{for each }i\text{ the set }U_{i}\cup (V_{i}\setminus C)\text{ is independent in }M_{i}.
\tag{$\clubsuit$}
\end{equation}

Observe first that if the condition $(\clubsuit)$ holds and there is an uncolored element $e$ ($e\in V_i$ for some $i$), then the player can make an `obvious' move, namely color $e$ with $i$. After this `obvious' move the condition $(\clubsuit)$ remains true.

To prove that Alice can keep the condition $(\clubsuit)$ assume that the condition holded (remember sets $U_i$ and $C$ of this moment) and later Bob has colored an element $e$ with color $j$. 

If $(U_{j}\cup e)\cup (V_{j}\setminus C)$ is independent in $M_{j}$, then $(\clubsuit)$ still holds and we use the above observation. 

When $(U_{j}\cup e)\cup (V_{j}\setminus C)$ is dependent in $M_{j}$, then by using the 
augmentation axiom (see Definition \ref{DefinitionIndependentSets}) we extend the independent set $U_{j}\cup e$ from the independent set $U_{j}\cup (V_{j}\setminus C)$ in $M_{j}$. The extension equals to $(U_{j}\cup e)\cup (V_{j}\setminus C)\setminus f$ for some $f\in V_{j}\setminus C$. Since sets $V_{1},\dots,V_{d}$ form a $2$-covering we know that $f\in V_{l}$ for some $l\neq j$. Now Alice has an admissible move, and her strategy is to color $f$ with color $l$. It is easy to observe that after her move the condition $(\clubsuit)$ is preserved.

To get the second assertion suppose $\chi(M)=k$, so matroid $M$ has a partition into independent sets $V_{1},\dots ,V_{k}$. We may take $M_{1}=\dots=M_{2k}=M$ and repeat twice each set $V_{i}$ to get a $2$-covering of $E$. By the first part of the assertion we infer that Alice has a winning strategy and therefore $\chi_{g}(M)\leq 2k$. 
\end{proof}

As usually in this kind of games the following question seems to be natural and non trivial (for graph coloring game it was asked by Zhu \cite{Zh99}).

\begin{question}
Suppose Alice has a winning strategy on a matroid $M$ with $k$ colors. Does she also have a winning strategy with $l>k$ colors?
\end{question}

We can define also a fractional version of the game chromatic number. Let $M$ be a matroid on a ground set $E$. For given integers $a$ and $b$ consider a slight modification of the matroid coloring game on $M$. As before, Alice and Bob alternately properly color elements of $E$ using a fixed set of $b$ colors. The only difference is that each element has to receive $a$ colors. The infimum of fractions $\frac{b}{a}$ for which Alice has a winning strategy is called the \emph{fractional game chromatic number} of $M$, which we denote by $\chi_{f,g}(M)$. Clearly, $\chi_{f,g}(M)\leq \chi_{g}(M)$ for every matroid $M$. Using Theorem \ref{TheoremGeneralizedGame} applied to $2b$ matroids equal to the $a$-th blow up of a loopless matroid $M$ (each element of $M$ is blown up to $a$ copies, see Definition \ref{DefinitionBlowUp}) we get that $\chi_{f,g}(M)\leq 2\chi_f(M)$. Motivated by the behaviour of fractional chromatic number we can ask two natural questions about fractional game chromatic number. The infimum in Definition \ref{DefinitionFractionalChromaticNumber} of fractional chromatic number is always reached (see the proof of Theorem \ref{TheoremFractionalChromaticNumber}). Is it also true in the case of fractional game chromatic number? The difference $\chi(M)-\chi_{f}(M)$ is less than $1$ (Theorem \ref{TheoremFractionalChromaticNumber}). Can the difference $\chi_{g}(M)-\chi_{f,g}(M)$ be arbitrary large? 

Now we generalize our results to the list version. The rules of the game between Alice and Bob do not change, except that now each element has its own list of colors. So instead of a fixed set of colors $C$, each element $e\in E$ can be colored by Alice or Bob only with a color from its list $L(e)$. The minimum number $k$ for which Alice has a winning strategy for every assignment of lists of size $k$ is called the \emph{game list chromatic number} of $M$, and denoted by $\ch_{g}(M)$. Clearly, $\ch_{g}(M)\geq\chi_{g}(M)$ for every matroid $M$, however the same upper bound as in Theorem \ref{TheoremGeneralizedGame} holds also for list parameter.

\begin{corollary}
For every loopless matroid $M$ we have $\ch_{g}(M)\leq 2\chi(M)$.
\end{corollary}

\begin{proof}
Let $\chi(M)=k$, and suppose $L:E\rightarrow\mathcal{P}(\{1,\dots,d\})$ is a list assignment of size $2k$. Denote $Q_i=\{e:i\in L(e)\}$ and $M_i=M\vert_{Q_i}$. By Theorem \ref{TheoremGeneralizationOfSeymour} for $\ell\equiv 2k$ and $w\equiv 2$, there is a $2$-covering of $E$ by sets $V_1,\dots,V_d$ with $V_i$ is independent in the matroid $M_i$. So by Theorem \ref{TheoremGeneralizedGame} Alice has a winning strategy, and as a consequence $\ch_g(M)\leq 2k$.
\end{proof}

In the same way as we defined  fractional game chromatic number one can define the fractional game list chromatic number $\ch_{f,g}(M)$. Clearly, $\ch_{f,g}(M)\leq\ch_{g}(M)$ and $\ch_{f,g}(M)\geq\chi_{f,g}(M)$ holds for every matroid $M$. As before, using Theorem \ref{TheoremGeneralizedGame} it is easy to get that $\ch_{f,g}(M)\leq 2\chi_f(M)$.

In view of Seymour's Theorem \ref{TheoremSeymour} we ask the following natural question.

\begin{question}
Does for every matroid $M$ the game list chromatic number equal to the game
chromatic number $\ch_g(M)=\chi_g(M)$?
\end{question}

Now we pass to the proof of a lower bound. We present a family of matroids $M_{k}$ for $k\geq 3$, with $\chi(M_k)=k$ and $\chi_{g}(M_k)\geq 2k-1$. This slightly improves the lower bound from the paper \cite{BaGrKi08} of Bartnicki, Grytczuk and Kierstead. They give an example of graphical matroids $H_{k}$ with $\chi(H_{k})=k$ and $\chi_{g}(H_{k})\geq 2k-2$ for each $k$. We believe that their family also satisfies $\chi_{g}(H_{k})\geq 2k-1$ for $k\geq 3$, however a proof would be much longer and more technical than in our case.

Fix $k\geq 3$. Let $E=C\sqcup D_1\sqcup\dots\sqcup D_{3k(2k-1)}$, where $C,D_1,\dots,D_{3k(2k-1)}$ are disjoint sets, $C=\{c_{1,1},\dots,c_{k,2k-1}\}$ has $k(2k-1)$ elements, and each $D_i=\{d_{1,i},\dots,d_{k,i}\}$ has $k$ elements. Let $M_k$ be a transversal matroid (see Definition \ref{DefinitionTransversalMatroid}) on a ground set $E$ with multiset of subsets $\mathcal{A}_k$ consisting of sets $D_1,\dots,D_{3k(2k-1)}$ and $(2k-1)$ copies of $E$. 

\begin{theorem}[Lasoń, \cite{La12}]\label{TheoremLowerBoundForGameColoring}
For every $k\geq 3$ matroid $M_{k}$ satisfies $\chi(M_k)=k$ and $\chi_{g}(M_{k})\geq 2k-1$.
\end{theorem}

\begin{proof}
To prove the first part of the assertion we can take $k$ independent sets $V_i=\{c_{i,1},\dots,c_{i,2k-1},d_{i,1},\dots,d_{i,3k(2k-1)}\}$. Observe that also $\chi_{f}(M_{k})=k$. 

To prove the second part notice that the rank of $C\cup D_i$ equals to $2k$, so if there are $t$ elements from $D_i$ colored with $i$, then there are at most $2k-t$ elements from $C$ colored with $i$. This suggests that Bob should try to color each $D_i$ with one color. 

Suppose the set of colors in the game is $\{1,\dots,h\}$ with $h\leq 2k-2$. We will show a winning strategy for Bob. Alice will always loose the game because she will not be able to color all elements of the set $C$. 

Assume first that Alice colors only elements from the set $C$, and her goal is to color all of them. It is the main case to understand. Denote by $d_i$ and $c_i$ the number of elements colored with $i$ in $D_i$ and $C$ respectively. Bob wants to keep the following invariant after each of his move:
\begin{equation}
d_i\geq c_i\text{ for any color }i.
\tag{$\heartsuit$}
\end{equation}
It is easy to see that he can always do it, and in fact there is an equality $d_i=c_i$ for every color $i$. Bob just mimics Alice's moves. Whenever she colors some $c\in C$ with $i$ he responds by coloring an element of $D_i$ with $i$. He can do it, because when after Alice's move $c_i=d_i+1 $ for some $i$, then $c_i+d_i\leq 2k$. But, then also $c_i+(d_i+1)\leq 2k$, so elements colored with $i$ are independent, and $d_i+1\leq k$, so there was an uncolored element in $D_i$. 

Observe that when Bob plays with this strategy, then for each color $i$ we have $c_i+d_i\leq 2k$ and $c_i\leq d_i$, so as a consequence $c_i\leq k$. This means that Alice can color only $hk$ elements of $C$, thus she looses.

It remains to justify that coloring elements of $D_1\cup\dots\cup D_{3k(2k-1)}$ by Alice can not help her in coloring elements of $C$. To see this we have to modify the invariant that Bob wants to keep. We assume $k\geq 4$, because for $k=3$ more careful case analysis is needed. Denote by $d_i,f_i$ the number of elements in $D_i\cup D_{i+h}\cup D_{i+2h}$ colored with $i$ and with some other color respectively, and $c_i$ as before. Now the invariant that Bob wants to keep after each of his move is the following: 
\begin{equation}
d_i\geq c_i+f_i\text{ or }d_i\geq k+2\text{ for any color }i.
\tag{$\spadesuit$}
\end{equation}
Analogously to the previous case one can show that Bob can keep this invariant. He just have to obey one more rule. Whenever Alice colors an element of $D_i\cup D_{i+h}\cup D_{i+2h}$, then Bob colors another element of this set with $i$ always trying to keep $\epsilon_i$, which is the number of sets among $D_i,D_{i+h},D_{i+2h}$ with at least one element colored with $i$, as low as possible. This completes a description of Bob's strategy. Observe that $f_i\geq\epsilon_i-1$.

Condition $(\spadesuit)$ gives the same consequence as $(\heartsuit)$. Let $\mathcal{D}_i$ be the union of those of $D_i,D_{i+h},D_{i+2h}$, which have an element colored with $i$. Matroid $M_k\vert_{C\cup\mathcal{D}_i}$ has rank $2k-1+\epsilon_i$. If $d_i\geq c_i+f_i$ we have 
$$2k-1+\epsilon_i\geq  c_i+d_i\geq 2c_i+f_i\geq 2c_i+\epsilon_i-1.$$ 
Otherwise, if $d_i\geq k+2$, then
$$2k+2\geq 2k-1+\epsilon_i\geq  c_i+d_i\geq c_i+k+2.$$ 
Therefore there can be at most $k$ elements in $C$ colored with $i$. 
\end{proof}

Suppose that Alice wins the game on $M_k$ (for $k\geq 3$), with $2k-1$ colors, and the above strategy for Bob. Then exactly $k$ elements of the set $C$ will be colored with $i$, for every color $i$. Additionally there are equalities in the inequalities showing a contradiction in the proof of Theorem \ref{TheoremLowerBoundForGameColoring}. So $M_k$ can be viewed as extremal examples for which $2k-1$ colors suffice for Alice to win.

Our results lead naturally to the following question.

\begin{question}
Is it true that $\chi_g(M)\leq 2\chi(M)-1$ holds for arbitrary loopless matroid $M$?
\end{question}

\section{On-line List Coloring}\label{SectionOnLineListColoring}

In this Section we consider a coloring of a matroid from lists (as in Section \ref{SectionListChromaticNumber} of Chapter \ref{ChapterColoringsOfMatroid}), but in a situation when only part of information about colors in the lists is known. This can be modeled by a two person game in which Alice will be coloring properly elements of the ground set of a matroid from lists revealed during the game by Bob. As usually Alice wins if at the end of the game, when all colors in the lists are revealed, the whole matroid is colored. 

Let $M$ be a loopless matroid on a ground set $E$, $\N_+=\{1,2,3,\dots\}$ be the set of colors and let $k$ be a positive integer. In the first round Bob chooses an arbitrary non-empty subset $B_1\subset E$ and inserts color $1$ to the lists of all elements of $B_1$. Then Alice decides which elements with $1$ on its list color with color $1$. So she chooses some independent set $A_{1}\subset B_{1}$ and colors its elements with $1$. In the $i$-th round Bob picks an arbitrarily non-empty subset $B_i\subset E$ and inserts color $i$ to the lists of all elements of $B_i$. Then Alice chooses an independent subset $A_i\subset B_i$ and colors its elements with color $i$. The game ends when all lists have exactly $k$ elements. If at the end of the game the whole matroid is colored, then Alice is the winner. Bob wins in the opposite case. 

Let $\ch_{on-line}(M)$ denote the minimum number $k$ for which Alice has a winning strategy in the above game. It is called the \emph{on-line list chromatic number} of $M$. Clearly, we have $\ch(M)\leq\ch_{on-line}(M)$ for every loopless matroid $M$.

The above game is a matroidal analog of the on-line graph coloring game introduced by Schauz \cite{Sc09} (see also \cite{Zh09}). For several classes of graphs on-line list chromatic number equals to list chromatic number. In almost all cases proofs of upper bounds for $\ch(G)$ remain valid for $\ch_{on-line}(G)$ \cite{Sc10'}. For example a simple modification of the argument of Thomassen \cite{Th94} shows that for every planar graph $G$ we have $\ch_{on-line}(G)\leq 5$. The same applies to the result of Galvin \cite{Ga95}, which works also in the on-line setting. Schauz proved \cite{Sc10} even an on-line version of combinatorial Nullstellensatz of Alon \cite{Al99} (we also generalize combinatorial Nullstellensatz in a different direction in \cite{La10}). It implies that if a bound $\ch(G)\leq k$ is a consequence of using combinatorial Nullstellensatz, then also $\ch_{on-line}(G)\leq k$ holds (for example $\ch_{on-line}(G)\leq 3$ for planar bipartite graph $G$, or asymptotic version of Dinitz's conjecture by H\"{a}ggkvist and Janssen \cite{HaJa97}). Recently, Gutowski \cite{Gu11} showed the on-line version of a theorem of Alon, Tuza and Voigt \cite{AlTuVo97}. Namely, that the fractional on-line list chromatic number equals to the fractional chromatic number. The only graphs for which it is known that parameters $\ch(G)$ and $\ch_{on-line}(G)$ differ is a certain family of graphs with $\ch(G)=2$ and $\ch_{on-line}(G)=3$ \cite{Zh09}. 

Surprisingly the best upper bound on $\ch_{on-line}(G)$ in terms of $\ch(G)$ is exponential as observed by Zhu \cite{Zh99}. It follows from a result of Alon \cite{Al00}. Therefore the problem of bringing upper and lower bounds closer is one of the most challenging problems in the area of on-line list chromatic number of graphs.

The main result of this Section, which comes from a joint paper with Wojciech Lubawski \cite{LaLu12}, asserts that for matroids we have equality between these two parameters. The proof relies on a multiple symmetric exchange property. Actually we prove a more general Theorem \ref{TheoremOnLineGeneralizationOfSeymour} that can be seen as the on-line version of our Generalization of Seymour's Theorem \ref{TheoremGeneralizationOfSeymour}.  

Recall that Theorem \ref{TheoremGeneralizationOfSeymour} provides equivalent conditions for a matroid $M$ to be $w$-colorable from any lists of size $\ell$. Now we may generalize this notion to the on-line setting in the same way as in the game of Schauz, except that now sizes of lists are exactly $\ell$ and Alice wins when she ended with a $w$-coloring of $M$. In this case we say that matroid $M$ is \emph{on-line $w$-colorable from lists of size $\ell$} if Alice has a winning strategy. 

Conditions from Theorem \ref{TheoremGeneralizationOfSeymour} are necessary for the above property to hold. Our aim is to prove that they are sufficient.

\begin{theorem}[Lasoń, Lubawski, \cite{LaLu12}] \emph{(On-line Version of Generalization of Seymour's Theorem)}\label{TheoremOnLineGeneralizationOfSeymour}
Let $\ell$ and $w$ be list size and weight functions on a matroid $M$ respectively. Then the following conditions are equivalent:
\begin{enumerate}
\item matroid $M$ is $w$-colorable from lists $L_{\ell}(e)=\{1,\dots,\ell(e)\}$,
\item matroid $M$ is $w$-colorable from any lists of size $\ell$,
\item matroid $M$ is on-line $w$-colorable from lists of size $\ell$.
\end{enumerate}
In particular, for every loopless matroid $M$ we have $\ch_{on-line}(M)=\chi(M)$.
\end{theorem}

\begin{proof}
Implications $(3)\Rightarrow (2)\Rightarrow (1)$ from the first part of the assertion are obvious. We will prove $(1)\Rightarrow (3)$ by induction on the number $w(E)=\sum_{e\in E}w(e)$. 

If $w(E)=0$, then $w$ is the zero vector and the assertion holds trivially. 

Suppose now that $w(E)\geq 1$ and the assertion of the theorem holds for all $w^{\prime }$ with $w^{\prime}(E)<w(E)$. Let $V\subset E$ be the set $B_1$ of elements picked by Bob in the first round of the game. So, all elements of $V$ have color $1$ in their lists. For a given subset $U\subset E$, denote by $c_U:E\rightarrow\{0,1\}$ the characteristic function of $U$, that is, $c_U(e)=1$ if $e\in U$, and $c_U(e)=0$, otherwise. 

To prove inductive step we have to show that there exists an independent set $I\subset V$ such that matroid $M$ is $(w-c_{I})$-colorable from lists $L_{\ell-c_{V}}(e)=\{1,2,\dots,\ell(e)-c_{V}(e)\}$. Then simply Alice takes $A_1=I$, so she colors elements of $I$ with color $1$.

By condition $(1)$ there exists a $w$-coloring $c:E\rightarrow\mathcal{P}(\N_+)$ of $M$ from lists $L_{\ell}(e)=\{1,\dots,\ell(e)\}$. Let $I_i$ denote the set of elements colored with $i$, for $i=1,\dots,k$. Let $X_{1}=V\cap I_{1}$. By Proposition \ref{PropositionMultipleIndependentSetExchange} (Multiple Symmetric Exchange Property) there exists $Y_{2}\subseteq I_{2}$ such that both $I_{1}^{\prime}=(I_{1}\setminus X_{1})\cup Y_{2}$ and $I_{2}^{\prime \prime}=(I_{2}\setminus Y_{2})\cup X_{1}$ are independent. In general let $X_{i}=V\cap I_{i}^{\prime \prime}$. So again by Proposition \ref{PropositionMultipleIndependentSetExchange} there exists $Y_{i+1}\subseteq I_{i+1}$, such that sets $I_{i}^{\prime }:=(I_{i}^{\prime \prime }\setminus X_{i})\cup Y_{i+1}$ and $I_{i+1}^{\prime \prime }:=(I_{i+1}\setminus Y_{i+1})\cup X_{i}$ are independent. Let $I=X_{k}$. 

It is not hard to see that $c':E\ni e\rightarrow\{i:e\in I_i^{\prime}\}$ is a $(w-c_{I})$-coloring from lists $L_{\ell-c_{V}}(e)=\{1,2,\dots,\ell(e)-c_{V}(e)\}$. The assertion of the theorem follows by induction.

The second part of the assertion follows as a corollary of the first one by taking $w\equiv 1$ and $\ell\equiv\chi(M)$. 
\end{proof}

Hence it turns out that matroids are so nice structures that if it is possible to color a matroid using $k$ colors, then it is not only possible to do it from any lists of size $k$, but one can even do it on-line.

\section{Indicated Coloring}\label{SectionIndicatedColoring}

In this Section we consider the following asymmetric variant of the coloring game, which was originally proposed by Grytczuk for graphs. In each round of the game Alice only indicates an uncolored element of the ground set of a loopless matroid $M$ to be colored. Then Bob chooses a color for that element from a fixed set of colors. The rule Bob has to obey is that it is a proper coloring. The goal of Alice is to achieve a proper coloring of the whole matroid. The least number of colors guaranteeing a win for Alice is denoted by $\chi_i(M)$ and called the \emph{indicated chromatic number} of $M$. Clearly $\chi_i(M)\geq\chi(M)$.

The indicated variant of the usual chromatic number for graphs, denoted by $\chi_i(G)$, was studied by Grzesik in \cite{Gr12}. It exhibits rather strange behaviour. For bipartite graphs we have trivially $\chi_i(G)=\chi(G)$, while there are $3$-chromatic graphs $G$ with $\chi_i(G)\geq 4$.

We show in Theorem \ref{TheoremGeneralIndicatedColoring} that for matroids there is always an equality $\chi_i(M)=\chi(M)$. The results of this Section come from our paper \cite{La13b}.

As before we will need a more general approach. Let $M_1,\dots,M_k$ be a collection of matroids on the same ground set $E$. We consider a modification of the above game in which the set of colors is $\{1,\dots,k\}$. The other rules are the same, except that now elements of color $i$ must form an independent set in the matroid $M_i$, for each $i=1,\dots,k$. 

\begin{theorem}[Lasoń, \cite{La13b}]\label{TheoremGeneralIndicatedColoring}
Let $M_1,\dots,M_k$ be matroids on the same ground set $E$. Suppose that there is a partition $E=V_1\cup\dots\cup V_k$ such that $V_i$ is independent in $M_i$. Then Alice has a winning strategy in the generalized indicated coloring game. In particular for every loopless matroid $M$ we have $\chi_i(M)=\chi(M)$.
\end{theorem}

\begin{proof}
The proof goes by induction on the number of elements of $E$. For a set $E$ consisting of only one element assertion clearly holds. Thus we assume that $\left\vert E\right\vert>1$ and that the assertion is true for all sets of smaller size. 

Denote the rank functions of matroids $M_1,\dots,M_k$ by $r_1,\dots,r_k$ respectively. From the Matroid Union Theorem \ref{TheoremMatroidUnion} we know that $$r_1(A)+\dots+r_k(A)\geq\left\vert A\right\vert$$ for each $A\subset E$. We distinguish two cases:

\begin{case2}
There is a proper subset $\emptyset\neq A\subsetneq E$ with the equality $$r_1(A)+\dots+r_k(A)=\left\vert A\right\vert.$$ 
\end{case2}

Then by subtracting equality for $A$ from the inequality for $A\cup B$ we get that for every subset $B\subset E\setminus A$ holds 
$$r_1(A\cup B)-r_1(A)+\dots+r_k(A\cup B)-r_k(A)\geq\left\vert A\cup B\right\vert-\left\vert A\right\vert=\left\vert B\right\vert.$$
The left hand side of this inequality is the sum of ranks of the set $B$ in matroids $M_1/A,\dots,M_k/A$. Thus by the Matroid Union Theorem \ref{TheoremMatroidUnion} the collection of matroids $M_1/A,\dots,M_k/A$ on $E\setminus A$ satisfies assumptions of the theorem. The collection $M_1\vert_{A},\dots,M_k\vert_{A}$ of matroids on $A$ clearly also does. By inductive assumption Alice has a winning strategy in the game on matroids $M_1\vert_{A},\dots,M_k\vert_{A}$ on the set $A$, so she plays with this strategy. Let $U_i$ denote the set of elements colored by Bob with $i$ after the game is finished. Now the play moves to the set $E\setminus A$, so the original game is on matroids $M_i/U_i$. But since the collection of matroids $M_i/A$ satisfies assumptions of the theorem, the collection of matroids $M_i/U_i$ also does, and Alice has a winning strategy by inductive assumption. As a result she wins the whole game on $E$. 

\begin{case2}
All proper subsets $\emptyset\neq A\subsetneq E$ satisfy the strict inequality $$r_1(A)+\dots+r_k(A)>\left\vert A\right\vert.$$ 
\end{case2}

Then in the first round of the game Alice indicates an arbitrary element $e\in E$. Obviously $e\in V_l$ for some $l$, thus Bob has an admissible move, namely he may color $e$ with $l$. Suppose Bob colors $e$ with color $j$. Now the original game is played on the matroids $M_i\vert_{E\setminus e}$ for $i\neq j$ and $M_j/e$, on the set $E\setminus e$. For these matroids the second condition of the Matroid Union Theorem \ref{TheoremMatroidUnion} holds since $r_j(A)$ can be smaller only by one. Hence Alice has a winning strategy by the inductive assumption. 
\end{proof} 

\part{Algebraic Problems on Matroids}

\chapter{$f$-Vectors of Simplicial Complexes}\label{ChapterSimplicialComplexes}

The results of this Chapter come from our paper \cite{La13a} and concern $f$-vectors of simplicial complexes. These vectors encode the number of faces of a given dimension in the complex, and are characterized by the Kruskal-Katona inequality. A similar characterization for matroids, or more generally, for pure complexes (those with all maximal faces, called facets, of the same dimension) remains elusive. 

A pure simplicial complex is called extremal if its $f$-vector satisfies equality in the highest dimensional Kruskal-Katona inequality. That is, a pure simplicial complex of dimension $d$ with $n$ facets is extremal if it has the least possible number of $(d-1)$-dimensional faces among all complexes with $n$ faces of dimension $d$. 

Our main result -- Theorem \ref{TheoremExtremalAreDecomposable} asserts that every extremal simplicial complex is vertex decomposable. This extends a theorem of Herzog and Hibi asserting that extremal complexes have the Cohen-Macaulay property. It is a simple fact that the last property is implied by vertex decomposability. 

Moreover, our argument is purely combinatorial with the main inspiration coming from the proof of Kruskal-Katona inequality. This answers a question of Herzog and Hibi, who asked for such a proof.

\section{Kruskal-Katona Inequality}\label{SectionKruskalKatonaInequality}

Simplicial complexes are one of basic mathematical structures, even more basic than matroids. They play important role in a number of mathematical areas, ranging from combinatorics to algebra and topology (see \cite{MiSt05,St96}). Here we present only most important information from our point of view.

Simplicial complexes can be viewed as discrete objects or as their geometric realizations. Our approach is combinatorial, thus we introduce abstract simplicial complex. 

\begin{definition}\emph{(Simplicial Complex)}\label{DefinitionSimplicialComplex}
A \emph{simplicial complex} $\Delta$ on a finite set $V$, called the set of \emph{vertices}, is a set of subsets of $V$, called \emph{faces} (or \emph{simplices}), satisfying property that a subset of a face is a face. 
\end{definition}

Cardinality of a face $\sigma$ minus one is called its \emph{dimension}, and it is denoted by $\dim\sigma$. The \emph{dimension} of a simplicial complex is the biggest dimension of its face. Maximal faces are called \emph{facets}, and a simplicial complex is called \emph{pure} if all its facets are of the same dimension. 

Matroids are regarded as simplicial complexes via their set of independent sets. Clearly, a matroid is a pure simplicial complex (see Definition \ref{DefinitionIndependentSets}), since all bases are of equal cardinality.

Let $\Delta$ be a simplicial complex of dimension $d$. Let $f_k$ be the number of $k$-dimensional faces of $\Delta$, for $k=-1,\dots,d$.  

\begin{definition}\emph{($f$-Vector)}\label{DefinitionfVector}
The \emph{$f$-vector} of a simplicial complex $\Delta$ is the sequence $(f_0,\dots,f_d)$.
\end{definition}

Observe that a simplicial complex with one face of dimension $k-1$ has at least $k$ faces of dimension $k-2$. One of most natural questions concerning simplicial complexes is to determine the minimum number of $(k-2)$-dimensional faces in a simplicial complex with $n$ faces of dimension $k-1$. A slightly more general problem is to characterize sequences which are $f$-vectors of some simplicial complex. Both problems were resolved independently by Kruskal \cite{Kr63} and Katona \cite{Ka68} in 1960's.

For a nonnegative integer $k$, they enlisted all $k$-element subsets of integers in the following order, called the \emph{squashed order}: 
\begin{equation*}
A<B\text{ if }\max (A\setminus B)<\max (B\setminus A).
\end{equation*} 
Let $S_{k}(n)$ be the set of first $n$ sets in this list. For a given set $\mathfrak{U}$ of $k$-element sets, denote by $\delta \mathfrak{U}$ the set of all $(k-1)$-element subsets of sets from $\mathfrak{U}$. Note that $\delta S_k(n)=S_{k-1}(m)$ for some $m$, that is subsets of fixed cardinality of sets from a prefix in squashed order form also a prefix in squashed order. The Kruskal-Katona theorem reads as follows.

\begin{theorem}[Kruskal, \cite{Kr63}; Katona, \cite{Ka68}] \emph{(Kruskal-Katona Inequality)}\label{TheoremKruskalKatona}
Let $\mathfrak{U}$ be a family of $n$ sets of size $k$. Then
\begin{equation*}
\left\vert {\delta\mathfrak{U}}\right\vert \geq \left\vert {\delta S_{k}(n)}\right\vert.
\end{equation*}
\end{theorem}

This result was further generalized by Clements and Lindstr\"{o}m in \cite{ClLi69}. Daykin \cite{Da74,Da75} gave two simple proofs, and later Hilton \cite{Hi79} gave another one. There is even an algebraic proof, for which we refer the reader to \cite{ArHeHi97}. 

The cardinality of the set $\delta S_{k}(n)$ may be easily determined. For a positive integer $k$, each positive integer $n$ can be uniquely expressed as 
\begin{equation*}
n=\displaystyle\binom{a_{k}}{k}+\displaystyle\binom{a_{k-1}}{k-1}+\dots+\displaystyle\binom{a_{t}}{t},
\end{equation*}
with $1\leq t\leq a_{t}$ and $a_{t}<\dots <a_{k}$. Denote by $\delta_{k-1}(n)=\left\vert {\delta S_{k}(n)}\right\vert$, we have 
\begin{equation*}
\delta_{k-1}(n)=\displaystyle%
\binom{a_{k}}{k-1}+\displaystyle\binom{a_{k-1}}{k-2}+\dots +\displaystyle%
\binom{a_{t}}{t-1}.
\end{equation*}

As one can easily see Kruskal-Katona Theorem \ref{TheoremKruskalKatona} gives the full characterization of $f$-vectors of simplicial complexes.

\begin{corollary} 
A sequence $(f_0,\dots,f_d)$ is the $f$-vector of a simplicial complex if and only if $f_{k-1}\geq\delta_{k}(f_k)$ for all $k=1,\dots,d$. 
\end{corollary}

A lot of work has been done to characterize $f$-vectors of pure simplicial complexes, but only some partial results are known (see \cite{BoMiMiNaZa12,PaZa12}). It is believed that they are too complicated to have an easy description. 

Even shape of $f$-vectors of matroids, which are a special class of pure simplicial complexes, is a mystery. There is no guess for a classification, only some conditions are expected to be necessary. Welsh \cite{We71} conjectured that $f$-vector of a matroid is unimodal, that is $f_k$ grows with $k$, and after reaching maximum value decreases ($f$-vector of a general pure simplicial complex need not be unimodal \cite{PaZa12}). Mason \cite{Ma72} strengthen it by conjecturing that $f$-vector of a matroid is log-concave, that is $f_k^2\geq f_{k-1}f_{k+1}$ for every $k$. Recently there was a breakthrough made by the seminal paper of Huh \cite{Hu11}. He proves that the coefficients of the characteristic polynomial of a matroid representable over a field of characteristic zero form a sign-alternating log-concave sequence. The proof uses advanced tools of algebraic geometry. Soon after it was generalized by Huh and Katz \cite{HuKa12} to representable matroids. Using this result Lenz \cite{Le11} proved Mason's conjecture for representable matroids. See also \cite{Hu12,Da84,St77,CoKaVa12} for conjectures and results concerning $h$-vector of a matroid -- another fundamental invariant of a simplicial complex encoding the number of faces in each dimension.

There are also classes of pure simplicial complexes for which a classification of $f$-vectors exists. For example, in the case of pure clique-complexes (faces are cliques contained in the graph) of chordal graphs (see \cite{HeHiSaTrZh08}).

Some research in the spirit of Kruskal-Katona theorem has been done also for other structures than simplicial complexes. Just to mention the result of Stanley (see \cite{St80,Fu93}), who proved using toric geometry a classification for convex polytopes.

\section{Vertex Decomposable Simplicial Complexes}\label{SectionVertexDecomposability}

\begin{definition}\emph{(Link of a Face)}
Let $\sigma $ be a face of a simplicial complex $\Delta$. The simplicial complex 
$$\{\tau \in \Delta :\tau \cap \sigma =\emptyset\text{ and }\tau\cup \sigma \in \Delta \}$$ 
is called the \emph{link of $\sigma $ in $\Delta$}, and denoted by $\lk_{\Delta}\sigma$. 
\end{definition}

For a vertex $x$ of $\Delta$ by $\Delta\setminus x$ we will denote the simplicial complex $\{\tau \in \Delta: x\notin\tau\}$. The definition of vertex decomposability is inductive. This notion was invented by Provan and Billera \cite{PrBi80}.

\begin{definition}\emph{(Vertex Decomposable Pure Simplicial Complex)}
A pure simplicial complex $\Delta$ is \emph{vertex decomposable} if one of the following conditions is satisfied:
\begin{enumerate}
\item $\Delta$ consists of empty set,
\item $\Delta$ consists of a single vertex,
\item for some vertex $x$ both $\lk_{\Delta }\{x\}$ and $\Delta\setminus x$ are pure and vertex decomposable.
\end{enumerate}
\end{definition}

In particular, a zero dimensional complex is vertex decomposable. It is a special case of a much more general fact.

\begin{proposition}\label{PropositionMatroidsVertexDecomposable}
Every matroid is vertex decomposable.
\end{proposition}

\begin{proof}
The proof goes by induction on the size of the ground set. If $\left\vert E\right\vert=1$, then matroid $M$ consists of a vertex or of the empty set. If $\left\vert E\right\vert>1$, then let $x$ be any element of $E$. Now both $\lk_M\{x\}=M/x$ and $M\setminus x$  are matroids on the ground set $E\setminus x$, hence the assertion follows by induction.
\end{proof}

We are ready to define the Stanley-Reisner ring of a simplicial complex.

\begin{definition}\emph{(Stanley-Reisner Ring)}\label{DefinitionStanleyReisnerRing}
Let $\Delta$ be a simplicial complex on the set of vertices $\{1,\dots,n\}$, and let $\K$ be a  field. The \emph{Stanley-Reisner ring} of $\Delta$ (or the \emph{face ring}) is 
$\K[\Delta]:=\K[x_{1},\dots x_{n}]/I$, where $I$ is the ideal generated by all square-free monomials $x_{i_{1}}\cdots x_{i_{l}}$ for which $\{i_{1},\dots,i_{l}\}$ is not a face in $\Delta$ (it is enough to consider only circuits of $\Delta$).
\end{definition}

When we say that a pure simplicial complex is \emph{Cohen-Macaulay} we always mean that its Stanley-Reisner ring has this property (see \cite{AtMa} for the definition of Cohen-Macaulay rings). Cohen-Macaulay simplicial complexes have also combinatorial description. Namely, Reisner \cite{Re76} proved that a pure simplicial complex $\Delta$ is Cohen-Macaulay over $\K$ if and only if for any face $\sigma\in\Delta$ all except top homologies of $\lk_{\Delta}\sigma$ over $\K$ are zero. For Cohen-Macaulay complexes we have a full classification of $h$-vectors. Macaulay \cite{Ma27} proved that they consist exactly of $O$-sequences (see e.g. \cite{BoMiMiNaZa12} for a definition of $O$-sequence).

It is a folklore result that vertex decomposability is a stronger property than being Cohen-Macaulay (we refer the reader e.g. to \cite{Bj95}, see also \cite[Section 3.2]{LaMi11} for an interesting method of counting homologies).

\begin{proposition}\label{PropositionBjorner}
Every vertex decomposable simplicial complex is Cohen-Macaulay over any field.
\end{proposition}

Due to Propositions \ref{PropositionMatroidsVertexDecomposable}, \ref{PropositionBjorner} and a theorem of Reisner for $\sigma=\emptyset$ we get the following interesting property of matroids (see also \cite{AhBe06,BjKoLo}). 

\begin{corollary}
All homologies of a matroid except top ones are zero.
\end{corollary}

\section{Extremal Simplicial Complexes}\label{SectionExtremalSimplicialComplexes}

\begin{definition}\emph{(Extremal Simplicial Complex)}
A pure simplicial complex $\Delta$ of dimension $d$ and $n$ facets is called \emph{extremal} if it has the least possible number of $(d-1)$-dimensional faces among all complexes with $n$ faces of dimension $d$. 
\end{definition}

We get the following as a consequence of Kruskal-Katona Theorem \ref{TheoremKruskalKatona}. 

\begin{corollary}\label{ExtremalfVector} 
A pure simplicial complex $\Delta$ of dimension $d>0$ with $f$-vector $(f_{0},\dots,f_{d})$
is extremal if and only if $f_{d-1}=\delta_{d}(f_{d})$. 
\end{corollary}

In particular, a zero dimensional complex is extremal, since all such complexes have exactly one $(-1)$-dimensional face, namely the empty set.

Herzog and Hibi proved using algebraic methods (in particular results from \cite{ArHeHi97,EaRe89}) the following theorem.

\begin{theorem}[Herzog, Hibi, \cite{HeHi99}]\label{ExtremalAreCohenMacaulay}
Every extremal simplicial complex is Cohen-Macaulay over any field.
\end{theorem}

They also asked for a combinatorial proof. We give it by proving by only combinatorial means that an extremal simplicial complex is vertex decomposable. Due to Proposition \ref{PropositionBjorner} vertex decomposable complex is Cohen-Macaulay over any field. 

\begin{theorem}[Lasoń, \cite{La13a}]\label{TheoremExtremalAreDecomposable}
Every extremal simplicial complex is vertex decomposable.
\end{theorem}

\begin{proof}
For a better understanding of extremal simplicial complexes we will use Hilton's idea from his proof \cite{Hi79} of Kruskal-Katona Theorem \ref{TheoremKruskalKatona}. First we define sets similar to $S_{k}(n)$. Let $S_{k}^{i}(n)$ be the first $n$ $k$-element subsets of integers in the squashed order ($A<B$ if $\max (A\setminus B)<\max (B\setminus A)$) which do not contain an element $i$. We will denote the set $\{i\cup A:A\in \mathfrak{U}\}$ by $i\uplus\mathfrak{U}$. 

Let $\mathfrak{U}$ be a family of $n$ $k$-element sets. Denote by $V=\bigcup\mathfrak{U}$ the underlying set of vertices, and by $v$ its cardinality. For $i\in V$, let $B_{i}=\{A\in\mathfrak{U}:i\notin A\}$, $C_{i}=\{A\setminus i:i\in A\in\mathfrak{U}\}$, and let $b_{i},c_{i}$ be the respective cardinalities. Note that always $c_{i}\neq 0$. We want to find a vertex $i$ such that $\left\vert {\delta B_{i}}\right\vert>\left\vert {C_{i}}\right\vert $.

\begin{lemma}\label{i} 
Either there exists a vertex $i$ such that $\left\vert {\delta B_{i}}\right\vert >\left\vert {C_{i}}\right\vert $, or $\mathfrak{U}$ consists of all possible $k$-element subsets of $V$.
\end{lemma}

\begin{proof}
We are going to count the sum of cardinalities of both sets when $i$ runs over all
elements of $V$. We have
\begin{equation*}
\sum_{i\in V}\left\vert {\delta B_{i}}\right\vert \geq kn(v-k)/(v-k)=kn=\sum_{i\in V}\left\vert {C_{i}}\right\vert, 
\end{equation*}%
since on the left hand side each $A\in\mathfrak{U}$ gives exactly $k$ distinct sets in its boundary, and each $A$ belongs to exactly $v-k$ sets $B_i$. Some sets in boundaries of sets from $B_{i}$
can be the same, but their number is at most $(v-1)-(k-1)=v-k$. On the right
hand side each $A\in\mathfrak{U}$ is counted $k$ times. Hence we can find a desired $i$, or
the above bounds are tight. In the latter case, if $A\in\delta B_{i}$, then all $v-k$ possibilities of completing it to a $k$-element set not containing $i$ have to be in $\mathfrak{U}$. If $v>k+1$ this means that $\mathfrak{U}$ consists of all possible $k$-element subsets of $V$. It is because from any set in $\mathfrak{U}$ we can delete any element and insert any other. If $v=k+1$ or $v=k$ the statement is easy to verify by hand.
\end{proof}

\begin{lemma}\label{characterization} 
If $\Delta $ is an extremal simplicial complex of positive dimension, then there exists a vertex $x$ such that both $\lk_{\Delta }\{x\}$ and $\Delta\setminus x$ are extremal.
\end{lemma}

\begin{proof}
Let $\Delta$ be an extremal simplicial complex of dimension $d-1>0$, and let $\mathfrak{U}$ be the set of all $d$-element faces in $\Delta$. 

If $\mathfrak{U}$ consists of all possible $d$-element subsets of a given $v$-element set, then the assertion of the lemma is clearly true. We can take any vertex. 

Otherwise, due to Lemma \ref{i}, there exists $i\in V$ such that $\left\vert {%
\delta B_{i}}\right\vert >\left\vert {C_{i}}\right\vert $. We have that 
\begin{equation*}
\delta\mathfrak{U}=\delta B_{i}\cup C_{i}\cup (i\uplus\delta C_{i}).
\end{equation*}%
Since $\delta B_{i}$ and $i\uplus\delta C_{i}$ are disjoint, it follows
that 
\begin{equation*}
\left\vert \delta\mathfrak{U}\right\vert \geq \left\vert \delta B_{i}\right\vert
+\left\vert i\uplus\delta C_{i}\right\vert > \left\vert {C_{i}}\right\vert
+\left\vert i\uplus\delta C_{i}\right\vert.
\end{equation*}
So, by Kruskal-Katona Theorem \ref{TheoremKruskalKatona}, 
\begin{equation}
\left\vert \delta\mathfrak{U}\right\vert \geq \left\vert \delta
S_{d}^{i}(b_{i})\right\vert +\left\vert i\uplus\delta
S_{d-1}^{i}(c_{i})\right\vert ,  
\tag{$\diamondsuit$}
\end{equation}
and
\begin{equation}
\left\vert \delta\mathfrak{U}\right\vert >\left\vert S_{d-1}^{i}(c_{i})\right\vert
+\left\vert i\uplus\delta S_{d-1}^{i}(c_{i})\right\vert.
 \tag{$\heartsuit$}
\end{equation}
Notice that $\delta S_{d}^{i}(b_{i})=S_{d-1}^{i}(e)$ for some $e$, thus $\delta S_{d}^{i}(b_{i})$ and $S_{d-1}^{i}(c_{i})$ are comparable in the inclusion order. If $\delta S_{d}^{i}(b_{i})\subset S_{d-1}^{i}(c_{i})$, then by $(\heartsuit)$ we get
$$\left\vert \delta\mathfrak{U}\right\vert >\left\vert S_{d-1}^{i}(c_{i})\right\vert
+\left\vert i\uplus\delta S_{d-1}^{i}(c_{i})\right\vert =\left\vert
\delta (S_{d}^{i}(b_{i})\cup (i\uplus S_{d-1}^{i}(c_{i})))\right\vert,$$
But this contradicts the assumption that $\Delta$ is extremal, since the complex generated by sets $S_{d}^{i}(b_{i})\cup (i\uplus S_{d-1}^{i}(c_{i}))$ has also $b_i+c_i=\left\vert\mathfrak{U}\right\vert$ maximal faces.

Thus $\delta S_{d}^{i}(b_{i})\supset S_{d-1}^{i}(c_{i})$, and by $(\diamondsuit)$ we obtain that
$$\left\vert \delta\mathfrak{U}\right\vert \geq \left\vert \delta
S_{d}^{i}(b_{i})\right\vert +\left\vert i\uplus\delta
S_{d-1}^{i}(c_{i})\right\vert =\left\vert \delta (S_{d}^{i}(b_{i})\cup
(i\uplus S_{d-1}^{i}(c_{i})))\right\vert.$$
With equality if and only if $C_{i}\subset \delta B_{i}$,$\;\left\vert
\delta B_{i}\right\vert =\left\vert \delta S_{d}^{i}(b_{i})\right\vert $,
$\;\left\vert \delta C_{i}\right\vert =\left\vert \delta
S_{d-1}^{i}(c_{i})\right\vert$.

Complex $\Delta $ is extremal, so there is an equality. We get that $C_{i}\subset \delta B_{i}$ and $(B_i),(C_i)$ are extremal, where $(A)$ denotes the simplicial complex generated by the set of faces $A$. Observe that $\lk_{\Delta }\{i\}=(C_i)$ and $\Delta\setminus i=(B_i)$. The first equality is obvious, while the second is not as clear. If $\sigma=\{v_1,\dots,v_k\}$ is a face in $\Delta\setminus i$, then it is a subface of some facet 
$F=\{v_1,\dots,v_d\}$. If $i$ does not belong to $F$ then $F\in (B_i)$ and so $\sigma$ does. Otherwise, $F\setminus \{i\}\in C_i\subset \delta B_i$, 
so $F\setminus i\cup j\in B_i$ for some $j$ and as a consequence $\sigma\in (B_i)$. Now take $x=i$ to get the assertion.
\end{proof}

The final step of the proof goes by induction on the dimension $d$ of $\Delta$ and secondly on the number of facets. If $d=0$, then $\Delta$ consists of points and almost by the definition it is vertex decomposable. When $d>0$, then by Lemma \ref{characterization} there exists a vertex $x$, such that both complexes $\lk_{\Delta }\{x\}$ and $\Delta\setminus x$ are extremal. The first is of lower dimension and the second either has the same dimension as $\Delta$ but fewer facets, or it has smaller dimension. By inductive hypothesis, both $\lk_{\Delta }\{x\}$ and $\Delta\setminus x$ are vertex decomposable, as a consequence $\Delta$ also is.
\end{proof}

We provide an example which shows that the opposite implication to the one from Theorem \ref{TheoremExtremalAreDecomposable} does not hold. 

\begin{example} 
A path on $n\geq 3$ vertices is a vertex decomposable simplicial complex which is not extremal.
\end{example}

Theorem \ref{TheoremExtremalAreDecomposable} can be treated as a characterization of $f$-vectors implying vertex decomposability. It is best possible in the following sense.  Let $\Delta$ be a pure simplicial complex of dimension $d>0$ with $f$-vector $(f_{0},\dots,f_{d})$, and with $f_{d-1}=\delta_{d}(f_{d})+c$, $c\in\N$. Due to Corollary \ref{ExtremalfVector} Theorem \ref{TheoremExtremalAreDecomposable} asserts that if $c=0$, then $\Delta$ is vertex decomposable. The example below shows that for $c=1$ it is already not true. 

\begin{example}
Let $\Delta$ be a pure simplicial complex of dimension $d$ with exactly two disjoint facets. If $\mathfrak{U}$ is the set of facets of $\Delta$, then $\left\vert\delta\mathfrak{U}\right\vert=2d+2$, while $\delta_{d}(2)=2d+1$. Complex $\Delta$ is not connected, thus due to a theorem of Reisner it is not Cohen-Macaulay over any field, and as a consequence also not vertex decomposable.
\end{example}

Unfortunately, there is no inclusion between matroids and extremal simplicial complexes. The simplicial complex generated by $S_2(4)$ is extremal, but it is not a matroid. The matroid $U_{1,1}\sqcup U_{1,3}$ (see Definitions \ref{DefinitionDirectSumOfMatroids}, \ref{DefinitionUniformMatroid}) is not an extremal simplicial complex.

\chapter{The Conjecture of White}\label{ChapterWhiteConjecture}

Describing minimal generating set of a toric ideal is a well-studied and difficult problem in commutative algebra. In 1980 White conjectured \cite{Wh80} that for every matroid $M$ its toric ideal $I_M$ is generated by quadratic binomials corresponding to symmetric exchanges. Despite big interest in algebraic combinatorics community the conjecture is still wide open. It is confirmed only for some very special classes of matroids like graphic matroids \cite{Bl08}, sparse paving matroids \cite{Bo13}, lattice path matroids (a subclass of transversal matroids) \cite{Sc11}, or matroids of rank at most three \cite{Ka10}.

In Theorem \ref{TheoremMain1} we prove White's conjecture for strongly base orderable matroids, which is already a large class of matroids. In particular it contains transversal matroids, for which it was known only that the toric ideal $I_M$ is generated by quadratic binomials \cite{Co07}, and also laminar matroids. 

However, the most important of our results is Theorem \ref{TheoremMain2} -- White’s conjecture up to saturation. We prove that the saturation with respect to the so-called irrelevant ideal of the ideal generated by quadratic binomials corresponding to symmetric exchanges equals to the toric ideal of a matroid. Our result has even more natural formulation in the language of algebraic geometry. Namely, it means that both ideals define the same projective scheme, while the conjecture asserts that their affine schemes are equal. To our knowledge it is the only result concerning White's conjecture, which is valid for all matroids.

In the last Section \ref{SectionRemarksWhiteConjecture} we discuss the original formulation of the conjecture due to White, which was combinatorial. It asserts that if two multisets of bases of a matroid have equal union (as a multiset), then one can pass between them by a sequence of single element symmetric exchanges. In fact White defined three properties of a matroid of growing difficulty, and conjectured that all matroids satisfy them. We reveal relations between these properties, present a beautiful alternative combinatorial reformulation by Blasiak \cite{Bl08}, and also mention how to extend our results to discrete polymatroids. 

The results of this Chapter come from the joint work \cite{LaMi13} with Mateusz Michałek.  

\section{Introduction}\label{SectionWhiteConjecture}

White's conjecture can be expressed in the language of combinatorics, algebra and geometry. We decided to present the conjecture and our results in the algebraic language, since then they have the most natural and compact formulation. However ideas of our proofs are purely combinatorial. 

Let $M$ be a matroid on a ground set $E$ with the set of bases $\mathcal{B}\subset\mathcal{P}(E)$. For a fixed field $\K$, let
$S_M:=\K[y_B:B\in\mathcal{B}]$
be a polynomial ring. There is a natural $\K$-homomorphism
$$\varphi_M:S_M\ni y_B\longrightarrow\prod_{e\in B}x_e\in\K[x_e:e\in E].$$

\begin{definition}\emph{(Toric Ideal of a Matroid)}
The \emph{toric ideal of a matroid} $M$, denoted by $I_M$, is the kernel of the map $\varphi_M$.
\end{definition}

Suppose that a pair of bases $D_1,D_2$ is obtained from a pair of bases $B_1,B_2$ by a symmetric exchange. That is, $D_1=B_1\cup f\setminus e$ and $D_2=B_2\cup e\setminus f$ for some $e\in B_1$ and $f\in B_2$. Then clearly the quadratic binomial
$$y_{B_1}y_{B_2}-y_{D_1}y_{D_2}$$
belongs to the ideal $I_M$. We say that such a binomial \emph{corresponds to symmetric exchange}.

\begin{conjecture}[White, \cite{Wh80}]\label{ConjectureWhite}
For every matroid $M$ its toric ideal $I_M$ is generated by quadratic binomials corresponding to symmetric exchanges. 
\end{conjecture}

Since every toric ideal is generated by binomials it is not hard to rephrase the above conjecture in the combinatorial language. It asserts that if two multisets of bases of a matroid have equal union (as a multiset), then one can pass between them by a sequence of single element symmetric exchanges. In fact this is the
original formulation due to White (cf. \cite{St97}). We immediately see that the conjecture does not depend on the field $\K$.

Actually White stated a bunch of conjectures of growing difficulty. The weakest asserts that the toric ideal $I_M$ is generated by quadratic binomials, the second one is Conjecture \ref{ConjectureWhite}, and the most difficult is an analog of Conjecture \ref{ConjectureWhite} for the noncommutative polynomial ring $S_M$. However, Conjecture \ref{ConjectureWhite} attracted the most attention because of its consequences in commutative algebra. We will discuss the original conjectures of White and show more subtle relations between them in Section \ref{SectionRemarksWhiteConjecture}. 

The most significant result towards Conjectue \ref{ConjectureWhite} was obtained by Blasiak in 2008 \cite{Bl08}, he confirmed it for graphical matroids. In 2010 Kashiwabara \cite{Ka10} checked the case of matroids of rank at most $3$. In 2011 Schweig \cite{Sc11} proved the case of lattice path matroids, which is a subclass of transversal matroids. Recently, in 2013 Bonin \cite{Bo13} solved the case of sparse paving matroids (all circuits have $r(M)+1$ or $r(M)$ elements). For transversal matroids it was known only, due to Conca in 2007 \cite{Co07}, that the toric ideal $I_M$ is generated by quadratic binomials. 

\section{White's Conjecture for Strongly Base Orderable Matroids}\label{SectionWhiteForSBO}

Here we prove White's conjecture for strongly base orderable matroids (see Definition \ref{DefinitionStronglyBaseOrderable}), which is already a large class of matroids. An advantage of this class of matroids is that it is characterized by a certain property instead of a specific presentation. Thus it is easier to check if a matroid is strongly base orderable or not. In particular we get that White's conjecture is true for transversal matroids, and also for laminar matroids.

Our argument uses an idea from the proof \cite{Sch03} of a theorem of Davies and McDiarmid \cite{DaMc76}. It asserts that if the ground set $E$ of two strongly base orderable matroids can be partitioned into $n$ bases in each of them, then there exists also a common partition.

\begin{theorem}[Lasoń, Michałek, \cite{LaMi13}]\label{TheoremMain1}
If $M$ is a strongly base orderable matroid, then the toric ideal $I_M$ is generated by quadratic binomials corresponding to symmetric exchanges. 
\end{theorem}

\begin{proof}
The ideal $I_M$, as every toric ideal, is generated by binomials. Thus it is enough to prove that quadratic binomials corresponding to symmetric exchanges generate all binomials of $I_M$.

Fix $n\geq 2$. We are going to show by decreasing induction on the overlap function
$$d(y_{B_1}\cdots y_{B_n},y_{D_1}\cdots y_{D_n})=\max_{\pi\in S_n}\sum_{i=1}^n\left\vert B_i\cap D_{\pi(i)}\right\vert$$ 
that a binomial $y_{B_1}\cdots y_{B_n}-y_{D_1}\cdots y_{D_n}\in I_M$ is generated by quadratic binomials corresponding to symmetric exchanges. Clearly, the biggest value of $d$ is $r(M)n$.

If $d(y_{B_1}\cdots y_{B_n},y_{D_1}\cdots y_{D_n})=r(M)n$, then there exists a permutation $\pi\in S_n$ such that for each $i$ holds $B_i=D_{\pi(i)}$, so $y_{B_1}\cdots y_{B_n}-y_{D_1}\cdots y_{D_n}=0$.

Suppose the assertion holds for all binomials with the overlap function greater than $d<r(M)n$. Let $y_{B_1}\cdots y_{B_n}-y_{D_1}\cdots y_{D_n}$ be a binomial of $I_M$ with the overlap function equal to $d$. Without loss of generality we can assume that the identity permutation realizes the maximum in the definition of the overlap function. Then there is $i$ for which there exists $e\in B_i\setminus D_i$. Clearly, $y_{B_1}\cdots y_{B_n}-y_{D_1}\cdots y_{D_n}\in I_M$ if and only if $B_1\cup\dots\cup B_n=D_1\cup\dots\cup D_n$ as multisets. Thus there exists $j\neq i$ such that $e\in D_j\setminus B_j$. Without lost of generality we can assume that $i=1,j=2$. Since $M$ is a strongly base orderable matroid, there exist bijections $\pi_B:B_1\rightarrow B_2$ and $\pi_D:D_1\rightarrow D_2$ with the multiple symmetric exchange property. We can assume that $\pi_B$ is the identity on $B_1\cap B_2$, and similarly that $\pi_D$ is the identity on $D_1\cap D_2$.

Let $G$ be a graph on a vertex set $B_1\cup B_2\cup D_1\cup D_2$ with edges $\{b,\pi_B(b)\}$ for all $b\in B_1\setminus B_2$ and $\{d,\pi_D(d)\}$ for all $d\in D_1\setminus D_2$. Graph $G$ is bipartite since it is a sum of two matchings. Split the vertex set of $G$ into two independent (in the graph sense) sets $S$ and $T$. Define:
$$B'_1=(S\cap (B_1\cup B_2))\cup (B_1\cap B_2),\;B'_2=(T\cap (B_1\cup B_2))\cup (B_1\cap B_2),$$
$$D'_1=(S\cap (D_1\cup D_2))\cup (D_1\cap D_2),\;D'_2=(T\cap (D_1\cup D_2))\cup (D_1\cap D_2).$$
By the multiple symmetric exchange property of $\pi_B$ sets $B'_1,B'_2$ are bases, and they are obtained from the pair $B_1,B_2$ by a sequence of symmetric exchanges. Therefore the binomial $y_{B_1}y_{B_2}y_{B_3}\cdots y_{B_n}-y_{B'_1}y_{B'_2}y_{B_3}\cdots y_{B_n}$ is generated by quadratic binomials corresponding to symmetric exchanges. Analogously the binomial $y_{D_1}y_{D_2}y_{D_3}\cdots y_{D_n}-y_{D'_1}y_{D'_2}y_{D_3}\cdots y_{D_n}$ is generated by quadratic binomials corresponding to symmetric exchanges. Moreover, since $S$ and $T$ are disjoint we have that 
$$d(y_{B'_1}y_{B'_2}y_{B_3}\cdots  y_{B_n},y_{D'_1}y_{D'_2}y_{D_3}\cdots  y_{D_n})>d(y_{B_1}y_{B_2}\cdots  y_{B_n},y_{D_1}y_{D_2}\cdots  y_{D_n}).$$ Thus by inductive assumption $y_{B'_1}y_{B'_2}y_{B_3}\cdots y_{B_n}-y_{D'_1}y_{D'_2}y_{D_3}\cdots y_{D_n}$ is generated by quadratic binomials corresponding to symmetric exchanges. Composing these three facts we get the inductive assertion.
\end{proof}

\section{White's Conjecture up to Saturation}\label{SectionWeakWhiteConjecture}

Let $\mathfrak{m}$ be the ideal generated by all variables in $S_M$ (so-called \emph{irrelevant ideal}). Recall that $I:\mathfrak{m}^{\infty}=\{a\in S_M:a\mathfrak{m}^n\subset I\text{ for some }n\in\N\}$ is called the \emph{saturation} of an ideal $I$ with respect to the ideal $\mathfrak{m}$.

Let $J_M$ be the ideal generated by quadratic binomials corresponding to symmetric exchanges. Clearly, $J_M\subset I_M$ and White's conjecture asserts that $J_M$ and $I_M$ are equal. We will prove a slightly weaker property, namely that ideals $J_M$ and $I_M$ are equal up to saturation, that is $J_M:\mathfrak{m}^\infty=I_M:\mathfrak{m}^\infty$  (cf. \cite[Section 4]{Mi12}). In fact the ideal $I_M$, as a prime ideal, is saturated so $I_M:\mathfrak{m}^\infty=I_M$. 

Now we will express our result in the language of algebraic geometry. From its point of view it is natural to compare schemes (main objects of study in algebraic geometry) defined by ideals instead of ideals themself. A homogeneous ideal ($J_M$ and $I_M$ are homogeneous, because every basis has the same number of elements) defines two kinds of schemes -- affine and projective. Affine schemes are equal if and only if the ideals are equal. Thus White's conjecture asserts equality of affine schemes of $J_M$ and $I_M$. Whereas ideals define the same projective scheme if and only if their saturations with respect to the irrelevant ideal are equal. That is, we prove that for every matroid $M$ projective schemes $\proj(S_M/J_M)$ and $\proj(S_M/I_M)$ are equal. Additionally, since both $J_M$ and $I_M$ are contained in $\mathfrak{m}$, this implies that both ideals have equal radicals, which in other words means they have the same affine set of zeros. The toric variety $\proj(S_M/I_M)$ (we refer the reader to \cite{Fu93,CoLiSc11} for background of toric geometry) has been already studied before. White proved that it is projectively normal \cite{Wh77}.

\begin{theorem}[Lasoń, Michałek, \cite{LaMi13}]\label{TheoremMain2}
White's conjecture is true up to saturation. That is, for every matroid $M$ we have $J_M:\mathfrak{m}^{\infty}=I_M$.
\end{theorem}

\begin{proof}
Since $J_M\subset I_M$ we get that $J_M:\mathfrak{m}^\infty\subset I_M:\mathfrak{m}^\infty=I_M$. 

To prove the opposite inclusion $I_M\subset J_M:\mathfrak{m}^\infty$ it is enough to consider binomials since $I_M$, as any toric ideal, is generated by binomials. To prove that a binomial $y_{B_1}\dots y_{B_n}-y_{D_1}\dots y_{D_n}\in I_M$ belongs to $J_M:\mathfrak{m}^\infty$ it is enough to show that for each basis $B\in\mathcal{B}$ we have
$$y_B^{(r(M)-1)n}(y_{B_1}\cdots y_{B_n}-y_{D_1}\cdots y_{D_n})\in J_M.$$ 
It is because then 
$$(y_{B_1}\cdots y_{B_n}-y_{D_1}\cdots y_{D_n})\mathfrak{m}^{(r(M)-1)n\left\vert\mathcal{B}\right\vert}\in J_M.$$ 

Let $B\in\mathcal{B}$ be a basis. The polynomial ring $S_M$ has a natural grading given by degree function $\deg(y_{B'})=1$, for each variable $y_{B'}$. We define the \emph{$B$-degree} by $\deg_B(y_{B'})=\left\vert B'\setminus B\right\vert$, and extend this notion also to bases $\deg_B(B')=\left\vert B'\setminus B\right\vert$. Notice that the ideal $I_M$ is homogeneous with respect to both gradings. Additionally $B$-degree of $y_B$ is zero, thus multiplying by $y_B$ does not change $B$-degree of a polynomial.  Observe that if $\deg_B(B')=1$, then $B'$ differs from $B$ only by a single element. We call such basis, and a corresponding variable, \emph{balanced}. A monomial or a binomial is called \emph{balanced} if it contains only balanced variables.

We will prove by induction on $n$ the following claim. As argued before this will finish the proof.

\begin{claim}\label{Claim1} If a binomial $b\in I_M$ has $B$-degree equal to $n$, then we have $$y_B^{\deg_B(b)-\deg(b)}b\in J_M.$$
\end{claim} 

If $\deg_B(b)-\deg(b)<0$, then the statement of the claim means that $y_B^{\deg(b)-\deg_B(b)}$ divides $b$, and that the quotient belongs to $J_M$. 

If $n=0$, then the claim is obvious, since $0$ is the only binomial in $I_M$ with $B$-degree equal to $0$. 

Suppose $n>0$. It is easier to work with balanced variables, therefore we provide the following lemma.  

\begin{lemma}\label{LemmaCommon}
For every basis $B'\in\mathcal{B}$ there exist bases $B_1,\dots,B_{\deg_B(B')}$ which are balanced and satisfy $$y_B^{\deg_B(B')-1}y_{B'}-y_{B_1}\cdots y_{B_{\deg_B(B')}}\in J_M.$$
\end{lemma}

\begin{proof} 
The proof goes by induction on $\deg_B(B')$. If $\deg_B(B')=0,1$, then the assertion is clear. If $\deg_B(B')>1$, then for $e\in B'\setminus B$ from the symmetric exchange property applied to $B,B'$ there exists $f\in B$ such that both $B_1=B\cup e\setminus f$ and $B''=B'\cup f\setminus e$ are bases. Now $\deg_B(B_1)=1$ and $\deg_B(B'')=\deg_B(B')-1$. Applying inductive assumption for $B''$ we get balanced bases $B_2,\dots,B_{\deg_B(B')}$ with $$y_B^{\deg_B(B')-2}y_{B''}-y_{B_2}\cdots y_{B_{\deg_B(B')}}\in J_M.$$
Hence, since $y_By_{B'}-y_{B_1}y_{B''}$ corresponds to symmetric exchange, we get
$$y_B^{\deg_B(B')-1}y_{B'}-y_{B_1}\cdots y_{B_{\deg_B(B')}}=y_B^{\deg_B(B')-2}(y_By_{B'}-y_{B_1}y_{B''})+$$ $$+y_{B_1}(y_B^{\deg_B(B')-2}y_{B''}-y_{B_2}\cdots y_{B_{\deg_B(B')}})\in J_M.$$
\end{proof}

Lemma \ref{LemmaCommon} allows us to change each factor $y_B^{\deg_B(B')-\deg(B')}y_{B'}$ of a monomial $y_B^{\deg_B(m)-\deg(m)}m$ to a product of balanced variables. Notice that the $B$-degree is preserved.
Hence Claim \ref{Claim1} reduces to the following claim.

\begin{claim}\label{Claim2} If $b=y_{B_1}\cdots y_{B_n}-y_{D_1}\cdots y_{D_n}\in I_M$ is a balanced binomial of $B$-degree equal to $n$, then $b\in J_M$.
\end{claim} 

With a balanced monomial $m=y_{B_1}\cdots y_{B_n}$ we associate a bipartite multigraph $G(m)$ with vertex classes $B$ and $E\setminus B$ (where $E$ is the ground set of matroid $M$). For each variable $y_{B_i}$ of monomial $m$ ($B_i=B\cup e\setminus f$ for $f\in B, e\in E\setminus B$) we put an edge $\{e,f\}$ in $G(m)$. 

Let $b=y_{B_1}\cdots y_{B_n}-y_{D_1}\cdots y_{D_n}\in I_M$ be a balanced binomial of $B$-degree equal to $n$. Observe that $b$ belongs to $I_M$ if and only if each vertex from $E$ has the same degree with respect to $G(y_{B_1}\cdots y_{B_n})$ and $G(y_{D_1}\cdots y_{D_n})$. Thus we can apply the following lemma, which is obvious from graph theory.

\begin{lemma}
Let $G$ and $H$ be bipartite multigraphs with the same vertex classes. Suppose that each vertex has the same degree with respect to $G$ and $H$. Then the symmetric difference of multisets of edges of $G$ and of $H$ can be partitioned into alternating cycles. That is simple cycles of even length, with two consecutive edges belonging to different graphs.
\end{lemma}

We get a simple cycle with vertices $f_1,e_1,f_2,e_2,\dots,f_r,e_r$ consecutively. So for each $i$ we have $B'_i=B\cup e_i\setminus f_i\in\mathcal{B}$ and $D'_i=B\cup e_{i-1}\setminus f_i\in\mathcal{B}$ (with a cyclic numeration modulo $r$). Then $y_{B'_1}\cdots y_{B'_r}$ divides $y_{B_1}\cdots y_{B_n}$ and $y_{D'_1}\cdots y_{D'_r}$ divides $y_{D_1}\cdots y_{D_n}$. 

Suppose $r<n$. The balanced binomial $b'=y_{B'_1}\cdots y_{B'_r}-y_{D'_1}\cdots y_{D'_r}$ belongs to $I_M$ and has $B$-degree less than $n$. From inductive assumption we get that $b'\in J_M$. There are balanced monomials $m_1,m_2,m_3$ such that
$$b=y_{B_1}\cdots y_{B_n}-y_{D_1}\cdots y_{D_n}=m_1b'-y_{D'_1}\cdots y_{D'_r}(m_2-m_3)$$
and $m_2-m_3\in I_M$. Balanced binomial $b''=m_2-m_3\in I_M$ has $B$-degree less than $n$, so by inductive assumption $b''\in J_M$, and as a consequence $b\in J_M$.

Suppose now that $r=n$. We can assume that $E=\{f_1,e_1\dots,f_n,e_n\}$, since otherwise we can contract $B\setminus\{f_1,\dots,f_n\}$ and restrict our matroid to the set $\{f_1,e_1\dots,f_n,e_n\}$. Obviously the property of being generated extends from such a minor to a matroid.

We say that a monomial $m_2$ is \emph{achievable} from $m_1$ if $m_1-m_2\in J_M$. In this situation we say also that variables of $m_2$ are \emph{achievable} from $m_1$. Observe that if there is a variable different from $y_B$ that is achievable from both monomials $y_{B_1}\cdots y_{B_n}$ and $y_{D_1}\cdots y_{D_n}$, then the assertion follows by induction. Indeed, if a variable $y_{B'}$ is achievable from both, then there are monomials $m_1,m_2$ such that $$b=(y_{B_1}\cdots y_{B_n}-y_{B'}m_1)+(y_{B'}m_2-y_{D_1}\cdots y_{D_n})+y_{B'}(m_1-m_2).$$
Binomial $b'=m_1-m_2\in I_M$ has $B$-degree less than $n$, thus by inductive assumption $y_B^{\deg_B(b')-\deg(b')}b'\in J_M$. Hence $b'\in J_M$ because $$\deg_B(b')-\deg(b')=\deg_B(b)-\deg(b)-\deg_B(y_{B'})+\deg(y_{B'})\leq 0.$$

Suppose contrary -- no variable different from $y_B$ is achievable from both monomials of $b$, and we will reach a contradiction. The forthcoming part of the proof illustrates a sentence of Sturmfels ``Combinatorics is a nanotechnology of mathematics''. For $k,i\in\Z_n$ (cyclic group $(\{1,\dots,n\},+($mod $n)$) we define:
$$S_k^i:=B\cup\{e_{k},e_{k+1},\dots,e_{k+i-1}\}\setminus\{f_k,f_{k+1},\dots,f_{k+i-1}\},$$
$$T_k^i:=B\cup\{e_{k-1},e_{k},\dots,e_{k+i-2}\}\setminus\{f_k,f_{k+1},\dots,f_{k+i-1}\},$$
$$U_k^i:=B\cup\{e_{k-i}\}\setminus\{f_k\}.$$
Sets $S_k^i$ and $T_k^i$ differ only on the set $\{e_1,\dots,e_n\}$ by a shift by one. Notice that $S_k^n=T_{k'}^n$ for arbitrary $k,k'\in\Z_n$ and $S_k^1=U_k^0=B'_k$, $T_k^1=U_k^1=D'_k$. Hence
$m_1=y_{B'_1}\cdots y_{B'_n}=y_{S_1^1}\cdots y_{S_n^1}$ and $m_2=y_{D'_1}\cdots y_{D'_n}=y_{T_1^1}\cdots y_{T_n^1}$ are monomials of $b$.

\begin{lemma}\label{LemmaNieMaStrzalek}
Suppose that for a fixed $0<i<n$ and every $k\in\Z_n$ the following conditions are satisfied:
\begin{enumerate}
\item sets $S_k^{i}$ and $T_k^{i}$ are bases,
\item monomial $y_B^{i-1}y_{S_k^{i}}\prod_{j\neq k,\dots,k+i-1} y_{S_j^{1}}$ is achievable from $m_1$,
\item monomial $y_B^{i-1}y_{T_k^i}\prod_{j\neq k,\dots,k+i-1} y_{T_j^1}$ is achievable from $m_2$.
\end{enumerate}
Then for every $k\in\Z_n$ neither of the sets $U_k^{-i}$, $U_k^{i+1}$ is a basis.
\end{lemma}

\begin{proof}
Suppose contrary that $U_k^{-i}$ is a basis. Then $y_{S_{k}^1}y_{S_{k+1}^i}-y_{T_{k+1}^i}y_{U_k^{-i}}$ by the definition belongs to $J_M$. Thus, by the assumption, variable $y_{T_{k+1}^i}$ would be achievable from both $m_1$ and $m_2$, which is a contradiction. The argument for $U_k^{i+1}$ is analogous, it differs by a shift by one.
\end{proof}

\begin{lemma}\label{LemmaSaNogi}
Suppose that for a fixed $0<i<n$ and every $k\in\Z_n$ the following conditions are satisfied:
\begin{enumerate}
\item set $S_k^{i}$ is a basis,
\item monomial $y_B^{i-1}y_{S_k^{i}}\prod_{j\neq k,\dots,k+i-1}y_{S_j^{1}}$ is achievable from $m_1$,
\item set $U_k^{-j}$ is not a basis for any $0<j\leq i$.
\end{enumerate}
Then for every $k\in\Z_n$ the set $S_k^{i+1}$ is a basis and $y_B^{i}y_{S_k^{i+1}}\prod_{j\neq k,\dots,k+i}y_{S_j^{1}}$ is a monomial achievable from $m_1$.
\end{lemma}

\begin{proof}
From the symmetric exchange property for $e_{k}\in S_{k}^1\setminus S_{k+1}^i$ it follows that there exists $x\in S_{k+1}^i\setminus S_{k}^1$ such that $\tilde S_{k+1}^i=S_{k+1}^i\cup e_{k}\setminus x$ and $\tilde S_{k}^1=S_{k}^1\cup x\setminus e_{k}$ are also bases. Thus $x\in\{f_{k},e_{k+1},e_{k+2},\dots,e_{k+i}\}$. Notice that if $x=e_{k+j}$ for some $j$, then $\tilde S_{k}^1=U_{k}^{-j}$ contradicting condition $(3)$. Thus $x=f_{k}$. Hence $\tilde S_{k+1}^i=S_{k}^{i+1}$ and $\tilde S_k^1=B$. In particular the binomial
$y_{S_{k+1}^i}y_{S_{k}^1}-y_By_{S_k^{i+1}}$ belongs to $J_M$ (condition $(1)$ guarantees that variable $y_{S_{k+1}^i}$ exists). Thus the assertion follows from condition $(2)$.
\end{proof}

Analogously we get the following shifted version of Lemma \ref{LemmaSaNogi}.

\begin{lemma}\label{LemmaSaNogiSym}
Suppose that for a fixed $0<i<n$ and every $k\in\Z_n$ the following conditions are satisfied:
\begin{enumerate}
\item set $T_k^{i}$ is a basis,
\item monomial $y_B^{i-1}y_{T_k^{i}}\prod_{j\neq k,\dots,k+i-1}y_{T_j^{1}}$ is achievable from $m_2$,
\item set $U_k^{j+1}$ is not a base for any $0<j\leq i$.
\end{enumerate}
Then for every $k\in\Z_n$ the set $T_k^{i+1}$ is a basis and $y_B^{i}y_{T_k^{i+1}}\prod_{j\neq k,\dots,k+i}y_{T_j^{1}}$ is a monomial achievable from $m_2$.
\end{lemma}

We are ready to reach a contradiction by an inductive argument. Firstly we verify that for $i=1$ the assumptions of Lemma \ref{LemmaNieMaStrzalek} are satisfied. Suppose now that for some $1<i<n$ the assumptions of Lemma \ref{LemmaNieMaStrzalek} are satisfied for every $1\leq j\leq i$. Then, by Lemma \ref{LemmaNieMaStrzalek} the assumptions of both Lemma \ref{LemmaSaNogi} and Lemma \ref{LemmaSaNogiSym} are satisfied for every $1\leq j \leq i$. Thus by the assertions of Lemmas \ref{LemmaSaNogi} and \ref{LemmaSaNogiSym}, the assumptions of Lemma \ref{LemmaNieMaStrzalek} are satisfied for all $1\leq j\leq i+1$.  We obtain that the assumptions and the assertions of Lemmas \ref{LemmaNieMaStrzalek}, \ref{LemmaSaNogi} and \ref{LemmaSaNogiSym} are satisfied for every $1\leq i<n$. For $i=n-1$ we obtain that the monomial $y_B^{n-1}y_{S_1^{n}}=y_B^{n-1}y_{T_1^{n}}$ is achievable from both $m_1$ and $m_2$, this gives a contradiction.
\end{proof}

\begin{corollary}
White's conjecture holds for matroid $M$ if and only if the ideal $J_M$ is saturated. 
\end{corollary}

In particular, in order to prove White's conjecture it is enough to show that the ideal $J_M$ is prime or radical.

\section{Remarks}\label{SectionRemarksWhiteConjecture}

We begin with the original formulation of conjectures stated by White in \cite{Wh80}. At the end we present a beautiful alternative combinatorial reformulation by Blasiak \cite{Bl08}, which is very natural, but not obviously equivalent.

Two sequences of bases $\mathsf{B}=(B_1,\dots,B_n)$ and $\mathsf{D}=(D_1,\dots,D_n)$ are \emph{compatible} if $B_1\cup\dots\cup B_n=D_1\cup\dots\cup D_n$ as multisets (that is if $y_{B_1}\cdots y_{B_n}-y_{D_1}\cdots y_{D_n}\in I_M$).

White defines three equivalence relations. Two sequences of bases $\mathsf{B}$ and $\mathsf{D}$ of equal length are in relation:
\newline
$\sim_1$ if $\mathsf{D}$ may be obtained from $\mathsf{B}$ by a composition of symmetric exchanges. That is $\sim_1$ is the transitive closure of the relation which exchanges a pair of bases $B_i,B_j$ in a sequence into a pair obtained by a symmetric exchange.
\newline $\sim_2$ if $\mathsf{D}$ may be obtained from $\mathsf{B}$ by a composition of symmetric exchanges and permutations of bases.
\newline $\sim_3$ if $\mathsf{D}$ may be obtained from $\mathsf{B}$ by a composition of multiple symmetric exchanges. 

Let $TE(i)$ denote the class of matroids such that for every $n\geq 2$ all compatible sequences of bases $\mathsf{B},\mathsf{D}$ of length $n$ satisfy $\mathsf{B}\sim_i\mathsf{D}$ (the notion $TE(i)$ is the same as original in \cite{Wh80}). An algebraic meaning of the property $TE(3)$ is that the toric ideal $I_M$ is generated by quadratic binomials. The property $TE(2)$ means that the toric ideal $I_M$ is generated by quadratic binomials corresponding to symmetric exchanges, while the property $TE(1)$ is its analog for the noncommutative polynomial ring $S_M$. 

We are ready to formulate the original conjecture \cite[Conjecture 12]{Wh80} of White.

\begin{conjecture}[White, \cite{Wh80}]\label{ConjectureOriginal}
We have the following equalities:
\begin{enumerate}
\item $TE(1)=$ the class of all matroids,
\item $TE(2)=$ the class of all matroids,
\item $TE(3)=$ the class of all matroids.
\end{enumerate}
\end{conjecture}

Clearly, Conjecture \ref{ConjectureWhite} coincides with Conjecture \ref{ConjectureOriginal} $(2)$. It is straightforward \cite[Proposition 5]{Wh80} that:
\begin{enumerate} 
\item $TE(1)\subset TE(2)\subset TE(3)$,
\item classes $TE(1),TE(2)$ and $TE(3)$ are closed under taking minors and dual matroid,
\item classes $TE(1)$ and $TE(3)$ are closed under direct sum.
\end{enumerate}

White claims also that the class $TE(2)$ is closed under direct sum, however unfortunately there is a gap in his proof. We believe that it is an open question. Corollary \ref{Corollary12} will show some consequences of $TE(2)$ being closed under direct sum to the relation between classes $TE(1)$ and $TE(2)$. 

\begin{lemma}\label{Lemma12}
The following conditions are equivalent for a matroid $M$:
\begin{enumerate}
\item $M\in TE(1)$,
\item $M\in TE(2)$ and for any two bases $(B_1,B_2)\sim_1 (B_2,B_1)$ holds.
\end{enumerate}
\end{lemma}

\begin{proof}
Implication $(1)\Rightarrow (2)$ is clear just from the definition. To get the opposite implication it is enough to recall that any permutation is a composition of transpositions.
\end{proof}

\begin{proposition}\label{Proposition12}
The following conditions are equivalent:
\begin{enumerate}
\item $M\in TE(1)$,
\item $M\sqcup M\in TE(1)$,
\item $M\sqcup M\in TE(2)$.
\end{enumerate}
\end{proposition}

\begin{proof}
Implications $(1)\Rightarrow (2)\Rightarrow (3)$ were already discussed. To get $(3)\Rightarrow (1)$ suppose that a matroid $M$ satisfies $M\sqcup M\in TE(2)$. 

First we prove that $M\in TE(2)$. Let $\mathsf{B}=(B_1,\dots,B_n)$ and $\mathsf{D}=(D_1,\dots,D_n)$ be compatible sequences of bases of $M$. If $B$ is a basis of $M$ then $\mathsf{B'}=([B_1,B],\dots)$ and $\mathsf{D'}=([D_1,B],\dots)$ are compatible sequences of bases of $M\sqcup M$ (where $[B',B'']$ denotes a basis of $M\sqcup M$ consisting of basis $B'$ of $M$ on the first copy and $B''$ on the second). From the assumption we have $\mathsf{B'}\sim_2\mathsf{D'}$. Notice that any symmetric exchange in $M\sqcup M$ is on the bases appearing on the first coordinate either trivial or a symmetric exchange. Thus, the same symmetric exchanges certify that $\mathsf{B}\sim_2\mathsf{D}$ in $M$. 

Now due to Lemma \ref{Lemma12} in order to prove $M\in TE(1)$ it is enough to show that for any two bases $B_1,B_2$ of $M$ relation $(B_1,B_2)\sim_1 (B_2,B_1)$ holds. Sequences of bases $([B_1,B_1],[B_2,B_2])$ and $([B_2,B_1],[B_1,B_2])$ in $M\sqcup M$ are compatible. Thus by the assumption one can be obtained from the other by a composition of symmetric exchanges and permutations. By the symmetry, without loss of generality we can assume that permutation are not needed. Now the projection of this symmetric exchanges to the first coordinate shows that $(B_1,B_2)\sim_1 (B_2,B_1)$ in $M$.
\end{proof}

As a corollary of Proposition \ref{Proposition12} we get that for reasonable classes of matroids the `standard' version of White's conjecture is equivalent to the `strong' one.

\begin{corollary}\label{Corollary12}
If a class of matroids $\mathfrak{C}$ is closed under direct sums, then $\mathfrak{C}\subset TE(1)$ if and only if $\mathfrak{C}\subset TE(2)$. In particular:
\begin{enumerate}
\item strongly base orderable matroids belong to $TE(1)$, 
\item Conjectures \ref{ConjectureOriginal} (1) and (2) are equivalent,
\item $TE(2)$ is closed under direct sum if and only if $TE(1)=TE(2)$. 
\end{enumerate}
\end{corollary}

Now we proceed to the alternative combinatorial description of White's conjectures by Blasiak \cite{Bl08}.

A matroid $M$ whose ground set $E$ is a union of two disjoint bases is called a \emph{$2$-matroid}. With a $2$-matroid $M$ we associate a graph $\mathfrak{B}_2(M)$ whose vertices are pairs of bases $(B_1,B_2)$ of $M$ whose union is $E$, and with edges between pairs of bases corresponding to symmetric exchange. This graph was already defined in 1985 by Farber, Richter and Shank \cite{FaRiSh85}, who conjectured that it is always connected. They proved it for graphic $2$-matroids (see also \cite{AnHoMe13} for a game version).

Blasiak defined an analog of graph $\mathfrak{B}_2(M)$ for matroids whose ground set is a disjoint union of more bases. Let $k\geq 3$. A matroid $M$ whose ground set $E$ is a union of $k$ disjoint bases is called a \emph{$k$-matroid}. With a $k$-matroid $M$ we associate a graph $\mathfrak{B}_k(M)$ whose vertices are sets of $k$ bases $\{B_1,\dots,B_k\}$ of $M$ whose union is $E$. An edge joins two vertices if they have non-empty intersection, that is if one of bases is in both partitions. Blasiak \cite{Bl08} proves the following.

\begin{proposition}
Conjecture \ref{ConjectureOriginal} (3) (`weak' version of White's conjecture) is equivalent to the fact that for every $k\geq 3$ and for every $k$-matroid $M$ the graph $\mathfrak{B}_k(M)$ is connected.

Conjecture \ref{ConjectureOriginal} (1) (`strong' version of White's conjecture) is equivalent to the fact that for every $k\geq 2$ and for every $k$-matroid $M$ the graph $\mathfrak{B}_k(M)$ is connected.
\end{proposition}

In the same way as we associate a toric ideal with a matroid one can associate a toric ideal $I_P$ with a discrete polymatroid $P$. Herzog and Hibi \cite{HeHi02} extend White's conjecture to discrete polymatroids, they also ask if the toric ideal $I_P$ of a discrete polymatroid possesses a quadratic Gr\"{o}bner basis (see also \cite{St95}). 

\begin{remark}
Theorem \ref{TheoremMain1} and Theorem \ref{TheoremMain2} are also true for discrete polymatroids. 
\end{remark}

There are several ways to prove that our results hold also for discrete polymatroids. One possibility is to use 
Lemma $5.4$ from \cite{HeHi02} to reduce a question if a binomial is generated by quadratic binomials corresponding to symmetric exchanges from a discrete polymatroid to a certain matroid. 
Another possibility is to associate to a discrete polymatroid $P\subset\Z^n$ a matroid $M_P$ on the ground set $\{1,\dots,r(P)\}\times\{1,\dots,n\}$ in which a set $I$ is independent if there is $v\in P$ such that $\left\vert I\cap\{1,\dots,r(P)\}\times\{i\}\right\vert\leq v_i$ holds for all $i$. It is straightforward that compatibility of sequences of bases and generation are the same in $P$ and in $M_P$. Moreover, one can easily prove that a symmetric exchange in $M_P$ corresponds to at most two symmetric exchanges in $P$.

\part{Applications of Matroids}

\chapter{Obstacles for Splitting Necklaces}

The results of this Chapter come from our paper \cite{La13c} and concern the famous necklace splitting problem. Suppose we are given a colored necklace, and we want to cut it so that the resulting pieces can be fairly split into two (or more) parts. What is the minimum number of cuts we have to make?

In the continuous version of the problem (from which a discrete one follows) a necklace is an interval, a coloring is a partition of a necklace into Lebesgue measurable sets, and a splitting is fair if each part captures the same measure of every color. A theorem of Goldberg and West \cite{GoWe85} asserts that if the number of parts is two, then every $k$-colored necklace has a fair splitting using at most $k$ cuts. This fact has a simple proof \cite{AlWe86} using the celebrated Borsuk-Ulam theorem (cf. \cite{Ma03}). Alon \cite{Al87} extended this result to arbitrary number of parts $q$ by showing that in this case $k(q-1)$ cuts suffice.

In a joint paper with Alon, Grytczuk, and Michałek \cite{AlGrLaMi09} we studied a kind of opposite question, which was motivated by the problem of Erd\H{o}s on strongly nonrepetitive sequences \cite{Er61} (cf. \cite{Cu93,De79,Gr08}). Namely, for which values of $k$ and $t$ there is a $k$-coloring of the real line such that no interval has a fair splitting into two parts with at most $t$ cuts? A theorem of Goldberg and West \cite{GoWe85} implies that $k>t$ is a necessary condition. We prove that $k>t+2$ is sufficient. 

In this Chapter we generalize this theorem to Euclidean spaces of arbitrary dimension $d$, and to arbitrary number of parts $q$. In higher dimensional problem a role of an interval plays an axis-aligned cube, and the cuts are made by axis-aligned hyperplanes. We prove that $k(q-1)>t+d+q-1$ is a sufficient condition (Theorem \ref{TheoremNecklaces}). Our bound is of the same order as a necessary condition $k(q-1)>t$, which is a consequence of a theorem of Alon \cite{Al87} (or by a more general theorem of de Longueville and \v{Z}ivaljevi\'{c} \cite{LoZi08}). This substantially improves a result of Grytczuk and Lubawski \cite{GrLu12}, which is exponential in $d$. Moreover for $d=1,q=2$ we get exactly the result of our previous paper \cite{AlGrLaMi09}. Additionally, we prove that if a stronger inequality $k(q-1)>dt+d+q-1$ is satisfied, then we can avoid necklaces in $\R^d$ with a fair $q$-splitting using arbitrary $t$ hyperplane cuts (Theorem \ref{TheoremNecklacesArbitrary}).

In an another result from \cite{AlGrLaMi09} of a slightly different flavor we proved that there is a $5$-coloring of the real line distinguishing intervals, that is no two intervals contain the same measure of every color (this solves a problem from \cite{GrSl03}). We generalize it by proving that there is a $(2d+3)$-coloring of $\R^d$ distinguishing axis-aligned cubes (Theorem \ref{2d+3}).

The main framework of our proofs is the topological Baire category theorem. However, the crucial part in the reasoning, which impacts on the bounds we achieve, is several application of algebraic matroids. 

\section{Introduction}

The following problem is well-known as the necklace splitting problem: \smallskip\newline
\emph{A necklace colored with $k$ colors has been stolen by $q$ thieves. They intend to share the necklace fairly, but they do not know what are the values of different colors. Therefore they want to cut it so that the resulting pieces can be fairly split into $q$ parts, which means that each part captures the same amount of every color. What is the minimum number of cuts they have to make?}\smallskip 

To be more precise a \emph{discrete $k$-colored necklace} is an interval in $\N$ colored with $k$ colors. A \emph{fair $q$-splitting of size $t$} is a division of the necklace using $t$ cuts into $t+1$ intervals, which can be split into $q$ parts with equal number of integers of every color. 

Goldberg and West \cite{GoWe85} proved that every discrete $k$-colored necklace with number of integers of every color divisible by $2$ has a fair $2$-splitting of size at most $k$ (see also \cite{AlWe86} for a short proof using the Borsuk-Ulam theorem, and \cite{Ma03} for other applications of the Borsuk-Ulam theorem in combinatorics). It is easy to see that this bound is tight. Consider a necklace consisting of $k$ pairs of consecutive integers, each pair colored differently. 

A famous generalization of Alon \cite{Al87} to the case of $q$ thieves, asserts that every discrete $k$-colored necklace with number of integers of every color divisible by $q$ has a fair $q$-splitting of size at most $k(q-1)$. A similar argument shows that the number of cuts is optimal.

All proofs of these discrete theorems go through a continuous version of the problem, in which a necklace is an interval, a coloring is a partition of a necklace into Lebesgue measurable sets, and a splitting is fair if each part captures the same measure of every color. 

Alon \cite{Al87} got even a more general version in which a coloring is replaced by a collection of arbitrary continuous probability measures (this generalizes a theorem of Hobby and Rice \cite{HoRi65}). 

\begin{theorem}[Alon, \cite{Al87}]\label{al}
Let $\mu_1,\dots,\mu_k$ be continuous probability measures on the interval $[0,1]$. Then there exists a division of  $[0,1]$ using at most $k(q-1)$ cuts into subintervals, which can be split into $q$ parts having equal measure $\mu_i$ for each $i$. 
\end{theorem}

Here we are interested in multidimensional necklace splitting problem. Notice that every axis-aligned hyperplane in $\R^d$ is perpendicular to exactly one of axes and parallel to others. Those perpendicular to $i$-th axis we call \emph{aligned to $i$-th axis}. Hence any collection of $t$ axis-aligned hyperplane cuts gives a natural partition $t=t_1+\dots+t_d$, where exactly $t_i$ hyperplanes are aligned to $i$-th axis. By projecting, Alon's Theorem \ref{al} implies that for any $k$ continuous probability measures on $[0,1]^d$ there exists a fair $q$-splitting using at most $k(q-1)$ hyperplanes aligned to $1$-st axis. De Longueville and \v{Z}ivaljevi\'{c} generalized it to a version in which we can fix arbitrary partition. 

\begin{theorem}[de Longueville, \v{Z}ivaljevi\'{c}, \cite{LoZi08}]\label{lozi}
Let $\mu_1,\dots,\mu_k$ be continuous probability measures on $d$-di\-men\-sion\-al cube $[0,1]^d$. For any selection of non-negative integers $t_1,\dots,t_d$ satisfying $k(q-1)=t_1+\dots+t_d$ there exists a division of $[0,1]^d$ using only axis-aligned hyperplane cuts and at most $t_i$ perpendicular to $i$-th axis into cuboids, which can be split into $q$ parts having equal measure $\mu_i$ for each $i$. 
\end{theorem}

For the case of arbitrary continuous probability measures it is also easy to see that the total number of hyperplanes needed, that is $k(q-1)$, is best possible. Indeed, just divide the diagonal of the cube into $k$ intervals and consider one dimensional measures on them. However, this argument has still one dimensional nature.

In this Chapter we restrict our attention from continuous probability measures to measurable colorings, that is we are interested in splitting fully colored $d$-dimensional necklaces. A \emph{$k$-colored $d$-dimensional necklace} is a nontrivial axis-aligned cube in $\R^d$ partitioned into $k$ Lebesgue measurable sets, which we call colors. A \emph{fair $q$-splitting of size $t$} of a necklace is a division using $t$ axis-aligned hyperplane cuts into cuboids, which can be split into $q$ parts, each capturing exactly $1/q$ of the total measure of every color. 

A very interesting question concerning splitting multidimensional necklaces is to find the minimum number $t_{d,k,q}$ of cuts such that any $k$-colored $d$-dimensional necklace has a fair $q$-splitting using at most $t_{d,k,q}$ axis-aligned hyperplane cuts. Clearly, Alon's Theorem \ref{al} (or more generally Theorem \ref{lozi} of de Longueville and \v{Z}ivaljevi\'{c}) implies that $k(q-1)\geq t_{d,k,q}$. In Theorem \ref{TheoremNecklace} we actually prove that for $q=2$ this bound is tight. Together with Theorem \ref{lozi}  this fully characterizes the set of tuples $(t_1,\dots,t_d)\in\N^d$, such that any $k$-colored $d$-dimensional necklace has a fair $2$-splitting using at most $t_i$ hyperplanes aligned to $i$-th axis for each $i$. Namely, it is the set $$\{(t_1,\dots,t_d)\in\N^d:\;k\leq t_1+\dots+t_d\}.$$

However, the most interesting problem is a bit more difficult. We want to find a measurable coloring of the whole space $\R^d$, which avoids necklaces having a fair $q$-splitting of a bounded size $t$. Of course, we want to minimize the number of colors $k$ we use. This problem was already formulated and partially solved in our joint paper \cite{AlGrLaMi09}. For $d=1,q=2$ a theorem of Goldberg and West \cite{GoWe85} (or more generally Theorem \ref{al}) implies that $k>t$ colors are needed, and we prove in \cite{AlGrLaMi09} that $k>t+2$ colors suffice.

\begin{theorem}[Alon, Grytczuk, Lasoń, Michałek, \cite{AlGrLaMi09}]\label{TheoremAGLM}
For every integers $k,t\geq 1$ if $k>t+2$, then there exists a measurable $k$-coloring of $\R$ such that no necklace has a fair $2$-splitting using at most $t$ cuts. 
\end{theorem}

Grytczuk and Lubawski \cite{GrLu12} generalized our result to higher dimensional spaces, by proving that the same holds in $\R^d$ (for $q=2$), provided 
$$k>\left(\frac{t}{d}+4\right)^d-\left(\frac{t}{d}+3\right)^d+\left(\frac{t}{d}+2\right)^d-2^d+t+2d+2.$$
They expected the lower bound on the number of colors is far from optimal, as even for $d=1$ it gives much worse result than the one from \cite{AlGrLaMi09}. Grytczuk and Lubawski asked \cite[Problem 9]{GrLu12} if $k\geq t+O(d)$ colors suffice for the case of two parts. 

In Theorem \ref{TheoremNecklaces} we prove that the answer is yes, namely 
$$k(q-1)>t+d+q-1$$ 
is a sufficient condition for $q$ parts (for $d=1$ we get even a better bound $k(q-1)>t+2$). Moreover, in this case the set of `good' colorings is dense. Note that our bound is of the same order as a necessary condition $k(q-1)>t$, which is a consequence of Theorem \ref{al} of Alon (or of a more general Theorem \ref{lozi} of de Longueville and \v{Z}ivaljevi\'{c}). For $d=1,q=2$ Theorem \ref{TheoremNecklaces} coincides with the result of our previous paper \cite{AlGrLaMi09}. We conjecture that for $q=2$ our bound is tight, as it captures the idea of degrees of freedom of objects we want to split.

In Section \ref{SectionArbitraryHyperplaneCuts} we generalize our results to the situation when arbitrary hyperplane cuts are allowed. This is closely related to the mass equipartitioning problem: 
\smallskip\newline
\emph{Does for every $k$ continuous mass distributions $\mu_1,\dots,\mu_k$ in $\R^d$ there exist $t$ hyperplanes dividing the space into pieces which can be split into $q$ parts of equal measure $\mu_i$ for each $i$?}\smallskip

A necessary condition for the positive answer is $k(q-1)\leq td$. It can be shown by considering one dimensional measure along the moment curve. The most basic and well-known positive result in this area is the Ham Sandwich Theorem for measures, asserting that for $(d,k,q,t)=(d,d,2,1)$ the answer is yes. It is a consequence of the Borsuk-Ulam theorem (see \cite{Ma03}), and has also discrete versions concerning equipartitions of a set of points in the space by a hyperplane. Several other things are known (see \cite{BlZi07,Ed87,MaVrZi06,Ma10,Ra96}), however even for some small values (eg. $(4,1,16,4)$) the question remains open. In the problem we are studying, $k$ continuous mass distributions in $\R^d$ are restricted to a measurable $k$-coloring of a given necklace. In this variant in Theorem \ref{TheoremNecklaceArbitrary} we reprove the necessary inequality for $q=2$. 

In the same spirit as in the case of axis-aligned hyperplane cuts we investigate existence of measurable colorings of $\R^d$ avoiding necklaces having a fair $q$-splitting of a bounded size $t$. In Theorem \ref{TheoremNecklacesArbitrary} we prove a sufficient condition on the number of colors $$k(q-1)>dt+d+q-1.$$

In the last Section \ref{SectionColoringsDistinguishingCubes} we generalize another result from \cite{AlGrLaMi09} of a slightly different flavor. In \cite{AlGrLaMi09} we proved that there is a measurable $5$-coloring of the real line distinguishing intervals, that is no two intervals contain the same measure of every color (this solves a problem from \cite{GrSl03}). In Theorem \ref{2d+3} we generalize it by proving that there is a measurable $(2d+3)$-coloring of $\R^d$ distinguishing axis-aligned cubes. We suspect that this number is optimal. Recently Vre\'{c}ica and \v{Z}ivaljevi\'{c} \cite{VrZi13} inspired by our question solved the case of arbitrary continuous probability measures. They showed $d+1$ measures distinguishing cubes, and proved that for any $d$ continuous probability measures on $\R^d$ there are two (in fact any finite number) non-trivial axis-aligned cubes which are not distinguished.

The proofs of our results are based on the topological Baire category theorem applied to the space of all measurable colorings of $\R^d$, equipped with a suitable metric. Section \ref{SectionTheSetting} describes this background. The remaining part of the argument is algebraic, it uses algebraic independence over suitably chosen fields. The rank of corresponding algebraic matroids describe well the idea of degrees of freedom of considered geometric objects. We will demonstrate in full details a proof of our main result -- Theorem \ref{TheoremNecklaces}. Proofs of other results fit the same pattern, so we will omit parts which can be rewritten with just minor changes and describe only main differences. 

\section{The Topological Setting}\label{SectionTheSetting}

Recall that a set in a metric space is \emph{nowhere dense} if the interior of its closure is empty. Sets which can be represented as a countable union of nowhere dense sets are said to be of \emph{first category}.

\begin{theorem}[Baire, cf. \cite{Ox80}]
If $X$ is a complete metric space and a set $A\subset X$ is of first category, then $X\setminus A$ is dense in $X$ (in particular is non-empty).
\end{theorem}

All of our results assert existence of certain `good' colorings. The main framework of the argument will mimic the one from \cite{AlGrLaMi09}. Our plan is to construct a suitable metric space of colorings and prove that the subset of `bad' colorings is of first category, resulting that `good' colorings exist. To prove it we will use entirely different technique than in our previous paper \cite{AlGrLaMi09}. Our approach will be purely algebraic, we will compare transcendence degrees of some sets of numbers over suitably chosen fields.

We go towards a proof of our main result -- Theorem \ref{TheoremNecklaces}, which asserts existence of a measurable $k$-coloring of $\R^d$ with no necklace having a fair $q$-splitting of size at most $t$, provided
$k(q-1)>t+d+q-1$. We will demonstrate this proof in full details. Proofs of other results fit the same pattern, so we will omit parts which can be rewritten with just minor changes and describe only main differences. This Section describes a common part of all proofs of our results.

Let $k$ be a fixed positive integer, and let $\{1,\dots,k\}$ be the set of colors. Let $M$ be a space of all measurable $k$-colorings of $\R^d$, that is the set of measurable maps $f:\R^d\rightarrow\{1,\dots,k\}$ (for each $i$ the set $f^{-1}(i)$ is Lebesgue measurable). For $f,g\in M$, and a positive integer $n$ let 
$$D_n(f,g)=\{x\in[-n,n]^d: f(x)\neq g(x)\}.$$
Clearly $D_n(f, g)$ is a measurable set, and we may define the normalized distance between $f$ and $g$ on $[-n,n]^d$ by 
$$d_n(f,g)=\frac{\mathcal{L}(D_n(f,g))}{(2n)^d},$$
where $\mathcal{L}(D)$ denotes the Lebesgue measure of a set $D$. We define the distance between any two colorings $f,g$ from $M$ by
$$d(f,g)=\sum_{n=1}^{\infty}\frac{d_n(f,g)}{2^n}.$$ 
To get a metric space we need to identify colorings whose distance is zero (they differ only on a set of measure zero), but they are indistinguishable from our point of view. Formally $\mathcal{M}$ is metric space consisting of equivalence classes of $M$ with distance function $d$. 

\begin{lemma}
Metric space $\mathcal{M}$ is complete.
\end{lemma}

This is a straightforward generalization of the fact that sets of finite measure in any measure space form a complete metric space, with measure of symmetric difference as the distance function (see \cite{Ox80}).

A coloring $f\in\mathcal{M}$ is called a \emph{cube coloring on $[-n,n]^d$} if there exists $N$ such that $f$ is constant on any half open cube from the division of $[-n,n]^d$ into $N^d$ equal size cubes. Let $\mathcal{I}_n$ denote the set of all colorings from $\mathcal{M}$ that are cube colorings on $[-n,n]^d$.

\begin{lemma}\label{aprox} 
Let $f\in\mathcal{M}$, then for every $\epsilon>0$ and an integer $n$ there exists a cube coloring $g\in\mathcal{I}_n$ such that $d(f,g)<\epsilon$.
\end{lemma}

\begin{proof} 
Let $C_i=f^{-1}(i)\cap[-n,n]^d$ and let $C_i^*\subset [-n, n]^d$ be a finite union of cuboids such that
$$\mathcal{L}((C_i\setminus C_i^*)\cup(C_i^*\setminus C_i))<\frac{\epsilon}{2k^2}$$
for each $i=1,\dots,k$. 
Define a coloring $h$ so that for each $i=1,\dots,k$, the set $C_i^*\setminus(C_1^*\cup\dots\cup C_{i-1}^*)$ is filled with color $i$, the remaining part of the cube $[-n,n]^d$ is filled with any fixed color, and $h$ agrees with $f$ everywhere outside $[-n,n]^d$. Then $d(f,h)<\epsilon/2$ and clearly each $h^{-1}(i)\cap[-n, n]^d$ is a finite union of cuboids. Let $A_1,\dots,A_t$ be the whole family of these cuboids. Now divide the cube $[-n, n]^d$ into $N^d$ equal cubes $B_1,\dots,B_{N^d}$ for $N>4dt(2n)^d/\epsilon$. Define a new coloring $g$ such that $g\vert_{B_i}\equiv h\vert_{B_i}$ whenever $B_i\subset A_j$ for some $j$, and $g\vert_{B_i}$ is constant equal to any color otherwise. Clearly $g$ is a cube coloring on $[-n,n]^d$. Moreover $g$ differs from $h$ on at most $2dt(N)^{d-1}$ small cubes of total volume at most $2dt(2n)^d/N<\epsilon/2$, thus we get $d(f,g)<\epsilon$.	
\end{proof}

\section{Axis-aligned Hyperplane Cuts}\label{SectionAxisAlignedHyperplaneCuts}

\begin{theorem}[Lasoń, \cite{La13c}]\label{TheoremNecklaces}
For every integers $d,k,t\geq 1$ and $q\geq 2$ if $k(q-1)>t+d+q-1$, then there exists a measurable $k$-coloring of $\R^d$ such that no $d$-dimensional necklace has a fair $q$-splitting using at most $t$ axis-aligned hyperplane cuts. The set of such colorings is dense. For $d=1$ the same holds provided $k(q-1)>t+2$. 
\end{theorem}

\begin{proof}
A splitting of $C=[\alpha_1,\alpha_1+\alpha_0]\times\dots\times[\alpha_d,\alpha_d+\alpha_0]$ using $t$ axis-aligned hyperplane cuts can be described by a partition $t_1+\dots+t_d=t$ such that for each $i$ there are exactly $t_i$ hyperplanes aligned to $i$-th axis, by $i$-th coordinates of this hyperplanes $\alpha_i=\beta^i_0\leq\beta^i_1\leq\dots\leq\beta^i_{t_i}\leq\beta^i_{t_i+1}=\alpha_i+\alpha_0$, and by a $q$-labeling of obtained pieces. By granularity of the splitting (or more generally of a division) we mean the length of the shortest subinterval $[\beta^i_j,\beta^i_{j+1}]$ over all $i,j$. 

For $(t_1,\dots,t_d)$ and $n$ we denote by $\mathcal{B}_n^{(t_i)}$ the set of $k$-colorings from $\mathcal{M}$ for which there exists at least one $d$-dimensional necklace contained in $[-n, n]^d$ which has a fair $q$-splitting using for each $i$ exactly $t_i$ hyperplanes aligned to $i$-th axis and with granularity at least $1/n$. Finally let $$\mathcal{B}_n=\bigcup_{t_1+\dots+t_d\leq t}\mathcal{B}_n^{(t_i)}$$
be a set of `bad' colorings. 

Let $\mathcal{G}_t$ be a subset of $\mathcal{M}$ consisting of those
$k$-colorings for which any necklace does not have a fair $q$-splitting using at most $t$ axis-aligned hyperplane cuts. Clearly we have
$$\mathcal{G}_t=\mathcal{M}\setminus\bigcup_{n=1}^{\infty}\mathcal{B}_n.$$

Our aim is to show that sets $\mathcal{B}_n$ are nowhere dense, provided suitable relation between $d,k,q$ and $t$ holds. This will imply that the union $\bigcup_{n=1}^{\infty}\mathcal{B}_n$ is of first category, so the set of `good' colorings $\mathcal{G}_t$, which we are looking for, is dense.

\begin{lemma}\label{LemmaClosed} Sets $\mathcal{B}_n$ are closed in $\mathcal{M}$.
\end{lemma}
\begin{proof} Since $\mathcal{B}_n$ is a finite union of $\mathcal{B}_n^{(t_i)}$ over all $t_1+\dots+t_d\leq t$ it is enough to show that sets $\mathcal{B}_n^{(t_i)}$ are closed. Let $(f_m)_{m\in\N}$ be a sequence of colorings from $\mathcal{B}_n^{(t_i)}$ converging  to $f$ in $\mathcal{M}$. For each $f_m$ there exists a necklace $C_m$ in $[-n,n]^d$ and a fair $q$-splitting of it with granularity at least $1/n$. Since $[-n,n]^d$ is compact we can assume that cubes $C_m$ converge to some cube $C$, and that hyperplane cuts also converge (notice that the granularity of this division of $C$ is at least $1/n$, in particular $C$ is a nontrivial). The only thing left is to define a $q$-labeling of pieces of the division to get $q$ equal parts. But there is a finite number of such labelings for each $C_m$, so one of them appears infinitely many times in the sequence. It is easy to see that this gives a fair $q$-splitting of the necklace $C$, so $f\in\mathcal{B}_n^{(t_i)}$.  
\end{proof}

\begin{lemma}\label{LemmaMain} 
If $k(q-1)>t+d+q-1$, then every $\mathcal{B}_n$ has empty interior. For $d=1$ the same holds provided $k(q-1)>t+2$. 
\end{lemma}
\begin{proof}
Suppose the assertion of the lemma is false: there is some $f\in\mathcal{B}_n$ and $\epsilon>0$ for which $K(f,\epsilon)=\{g\in\mathcal{M}:d(f,g)<\epsilon\}\subset\mathcal{B}_n$. By Lemma \ref{aprox} there is a cube coloring $g\in\mathcal{I}_n$ such that $d(f,g)<\epsilon/2$, so $K(g,\epsilon/2)\subset\mathcal{B}_n$. 

The idea is to modify slightly the coloring $g$ so that the new coloring will be still close to g, but there will be no necklaces inside $[-n,n]^d$ possessing a $q$-splitting of size at most $t$ and granularity at least $1/n$. 

Without loss of generality we may assume that there are equally spaced axis-aligned cubes $C_{i_1,\dots,i_d}$ for $i_1,\dots,i_d\in\{1,\dots,N\}$ in $[-n,n]^d$ such that $N > 4n^2$, and each cube is filled with a unique color in the cube coloring $g$. Notice that edges of cubes are rational, equal to $\frac{2n}{N}$. Let $\delta > 0$ be a rational number satisfying $\delta^d<\min\{\frac{\epsilon}{2N^d},\left(\frac{2n}{N}\right)^d\}$. Inside each cube $C_{i_1,\dots,i_d}$ choose a volume $\delta^d$ axis-aligned cube $D_{i_1,\dots,i_d}$ with corners in rational points $\Q^d$.  Let us call `white' the first color from our set of colors $\{1,\dots,k\}$. For each cube $D_{i_1,\dots,i_d}$ and a color $j=2,\dots,k$ we choose small enough real number $x_{j;i_1,\dots,i_d}>0$ and an axis-aligned cube $D_{j;i_1,\dots,i_d}\subset D_{i_1,\dots,i_d}$ such that:
\begin{enumerate}
\item length of the side of $D_{j;i_1,\dots,i_d}$ equals to $x_{j;i_1,\dots,i_d}$
\item cubes $D_{j;i_1,\dots,i_d}$ for $j=2,\dots,k$ are disjoint
\item one of corners of $D_{j;i_1,\dots,i_d}$ is rational (lies in $\Q^d$)
\item the set $\{x_{j;i_1,\dots,i_d}:j\in\{2,\dots,k\},i_1,\dots,i_d\in\{1,\dots,N\}\}$ is algebraically independent over $\Q$.
\end{enumerate}

Now we define a coloring $h$, which is a modification of $g$. Let $h$ equals to $g$ outside cubes $D_{i_1,\dots,i_d}$ for $i_1,\dots,i_d\in\{1,\dots,N\}$. Inside the cube $D_{i_1,\dots,i_d}$ let $h$ be equal to `white' everywhere except cubes $D_{j;i_1,\dots,i_d}$, where it is $j$. 

Clearly $d(g,h)<\epsilon/2$, so $h\in\mathcal{B}_n$. Hence there is a necklace $A=[\alpha_1,\alpha_1+\alpha_0]\times\dots\times[\alpha_d,\alpha_d+\alpha_0]\subset [-n,n]^d$ which has a fair $q$-splitting of size at most $t$ and granularity at least $1/n$. Denote by $\beta_1,\dots,\beta_t$ the coordinates of axis-aligned hyperplane cuts (for $H=\{x_i=\beta\}$ we get $\beta$). 

Since granularity of the splitting is at least $1/n$, and edges of cubes $C_{i_1,\dots,i_d}$ are smaller or equal to $1/(2n)$, each piece $P$ of the division contains a cube $C_{i_1,\dots,i_d}\subset P$ for some $i_1,\dots,i_d$ (so in particular $C_{i_1,\dots,i_d}$ is not divided by any hyperplane cut). Each of $q$ parts contains at least one piece and as a consequence at least one cube. For each part we fix such a cube $C_{i^l_1,\dots,i^l_d}$ for $l=1,\dots,q$.

Now the key observation is that the amount of color $j$ different from `white' contained in each part $l$ of the $q$-splitting is a polynomial $W_{j,l}$ from $$\Q[\alpha_i,\beta_m,x_{j;i_1,\dots,i_d}:\;i=0,\dots,d;\;m=1,\dots,t;\;(i_1,\dots,i_d)\in\{1,\dots,N\}^d].$$ 
Comparing the amount of color $j$ in the $l$-th part and the first one we get
$$W_{j,l}(\alpha_i,\beta_m,x_{j;i_1,\dots,i_d})=W_{j,1}(\alpha_i,\beta_m,x_{j;i_1,\dots,i_d}).$$
Since the cube $C_{i^l_1,\dots,i^l_d}$ is not divided by any hyperplane from our splitting, a variable $x_{j;i^l_1,\dots,i^l_d}$ appears only in the component $x^d_{j;i^l_1,\dots,i^l_d}$ in the polynomial $W_{j,l}$. Denote $W_{j,l}'=W_{j,l}-x^d_{j;i^l_1,\dots,i^l_d}$. We get polynomial equations $(\spadesuit)$:
$$(W_{j,1}-W_{j,l}')(\alpha_i,\beta_m,x_{j;i_1,\dots,i_d}:\; (i_1,\dots)\neq (i^r_1,\dots)\; r=2,\dots,q)=x^d_{j;i^l_1,\dots,i^l_d},$$
with coefficients in $\Q$, for all $j=2,\dots,k$ and $l=2,\dots,q$.

Let $M_{\K}$ be the algebraic matroid for a field extension $$\K\subset\Le=\Q(\alpha_i,\beta_m,x_{j;i_1,\dots,i_d})$$ and a suitable (big enough) finite subset $E\subset\Le$.

Denote by $T$ the set $\{\alpha_i,\beta_m:\;i=0,\dots,d;\;m=1,\dots,t\}$, and suppose its rank in the matroid $M_{\Q}$ equals to $t+d+1-c$, for some $c$. We will show that $c\geq 1$, and $c\geq q-1$ for $d=1$ (notice that in the case $d=1$ we can assume, without loss of generality, that $t\geq q-1$). Observe that the volume of each part of the splitting is a polynomial over $\Q$ in variables $\alpha_0,\dots\alpha_d,\beta_1,\dots,\beta_t$. All $q$ parts have to be of the same volume, so we get $q-1$ nontrivial polynomial equations. Each of them shows that $T$ is a dependent set in $M_{\Q}$, hence $c\geq 1$. For us the ideal situation would be if equations were `independent' leading to $c=q-1$, since it would give a better lower bound $k(q-1)>t+d+1$. Unfortunately, it is not always the case, see Example \ref{q-1}. However in dimension one these equations are linear, and it is easy to see that they are linearly independent. Hence $c\geq q-1$ for $d=1$. Denote by $\K$ the field
$$\Q(x_{j;i_1,\dots,i_d}:\;(i_1,\dots,i_d)\neq (i^r_1,\dots,i^r_d)\; r=2,\dots,q).$$
Clearly, inclusion $\Q\subset\K$ implies that $r_{M_{\K}}(T)\leq r_{M_{\Q}}(T)=t+d+1-c$. 

Denote by $L$ the set of left sides of equations $(\spadesuit)$, and by $R$ the set of right sides of equations $(\spadesuit)$. We have that $L\subset\cl_{M_{\K}}(T)$ (see Definition \ref{DefinitionClosure} of closure), so $r_{M_{\K}}(L)\leq r_{M_{\K}}(T)=t+d+1-c$. From condition $(4)$ of the choice of $x_{j;i_1,\dots,i_d}$ it follows that the set $R$ is independent in $M_{\K}$, that is $r_{M_{\K}}(R)=(k-1)(q-1)$. But obviously $L=R$, hence $(k-1)(q-1)\leq t+d$ (and $(k-1)(q-1)\leq t-q+3$ for $d=1$) and we get a contradiction.
\end{proof}

Now from Lemmas \ref{LemmaClosed} and \ref{LemmaMain} it follows that if $k(q-1)>t+d+q-1$, or $k(q-1)>t+2$ for $d=1$, then sets $\mathcal{B}_n$ are nowhere dense, which proves the assertion of the theorem.
\end{proof}

\begin{example}\label{q-1}
Consider $d=2,q=4,t_1=1,t_2=1$. Then equality of volumes of three parts of the splitting implies that the fourth one also has the same volume. This phenomena is caused by the geometry of $\R^d$ and axis-aligned hyperplanes. It shows that in general it is easier to $q$-split fairly a whole necklace, than its Lebesgue measurable subset. In dimension one these situations do not happen.
\end{example}

By modifying slightly the proof of Theorem \ref{TheoremNecklaces} one can easily obtain a version for cuboids with inequality $k(q-1)>t+2d+q-2$.

The idea standing behind the proof is to count `degrees of freedom' of moving a necklace and cuts splitting it and to compare it with the number of equations forced by existence of a fair splitting. If the number of degrees of freedom is less than the number of equations, then there exists a coloring with the property that no necklace can be fairly split. Even though the intuitive meaning of degree of freedom is clear, it is hard to define it rigorously to be able to make use of it. It turns out that the rank in some algebraic matroid of the set $\{\alpha_i,\beta_m:\;i=0,\dots,d;\;m=1,\dots,t\}$ works. 

As we will see in Section \ref{SectionArbitraryHyperplaneCuts}, if we allow arbitrary hyperplane cuts, then each such cut adds $d$ degrees of freedom, not only one as an axis-aligned hyperplane cut. Therefore for arbitrary hyperplane cuts we will get roughly $d$ times worse bound. The same reason explains the difference between results for cubes and cuboids, the latter have larger degree of freedom.

For $d=1$ and $q=2$ Theorem \ref{TheoremNecklaces} gives exactly Theorem \ref{TheoremAGLM}, the result of our joint paper with Alon, Grytczuk and Michałek \cite{AlGrLaMi09}. However, the technique of showing that sets $\mathcal{B}_n$ have empty interior is quite different. To show a contradiction we used in \cite{AlGrLaMi09} only one color, which was different from colors of cut points and the end points of an interval. If $d>1$ or $q>2$ this is not always possible. Additionally our previous argument was not only algebraic, but contained also an analytic part. We used some inequalities on measure of colors. We ask if our bound from Theorem \ref{TheoremNecklaces} for $q=2$ is tight. 

\begin{question}
Fix $d,k,t\geq 1$ such that $k\leq t+d+1$. Does for every measurable $k$-coloring of $\R^d$ there exist a $d$-dimensional necklace which has a fair $2$-splitting using at most $t$ axis-aligned hyperplane cuts?
\end{question}

We suspect that the answer is positive, however we know that if we replace a $k$-coloring by arbitrary $k$ continuous measures then it is negative. In Section \ref{SectionColoringsDistinguishingCubes} we provide other examples when the case of arbitrary continuous measures differs from the case of measurable coloring.

The last part of this Section is devoted to colorings of a given fixed necklace that avoid a fair splitting of a certain size.

\begin{theorem}[Lasoń, \cite{La13c}]\label{TheoremNecklace}
For every integers $d,k,t\geq 1$ and $q\geq 2$ if $k(q-1)>t+q-2$, then there is a measurable $k$-coloring of a $d$-dimensional necklace for which there is no fair $q$-splitting using at most $t$ axis-aligned hyperplane cuts. The set of such colorings is dense. 
\end{theorem}

\begin{proof}
For $d,k\geq 1$, $q\geq 2$ and a $d$-dimensional necklace $C=[\alpha_1,\alpha_1+\alpha_0]\times\dots\times[\alpha_d,\alpha_d+\alpha_0]$ we define the set $\mathcal{B}_n^{(t_i)}$ of $k$-colorings from $\mathcal{M}$ for which there is a fair $q$-splitting of $C$ with exactly $t_i$ hyperplanes aligned to $i$-th axis, for each $i$, and granularity at least $1/n$. Finally let
$$\mathcal{B}_n=\bigcup_{t_1+\dots+t_d\leq t}\mathcal{B}_n^{(t_i)}$$
be `bad' colorings. The aim, as previously, is to show that sets $\mathcal{B}_n$ are nowhere dense, provided suitable relation between $d,k,q$ and $t$ holds. This will imply that the union $\bigcup_{n=1}^{\infty}\mathcal{B}_n$ is of first category, so the set of colorings which we are looking for is dense. Analogously to Lemma \ref{LemmaClosed} sets $\mathcal{B}_n$ are closed in $\mathcal{M}$. Similarly to Lemma \ref{LemmaMain} if $k(q-1)>t+q-2$, then $\mathcal{B}_n$ has empty interior. The only difference is that we want the set 
$$\{x_{j;i_1,\dots,i_d}:j\in\{2,\dots,k\},i_1,\dots,i_d\in\{1,\dots,N\}\}$$ 
to be algebraically independent over $\Q(\alpha_i:i=0,\dots,d)$. We also add $\alpha_0,\dots,\alpha_d$ to the base field $\K$, so the rank in the matroid of the left side of equations is equal to $t-c\leq t-1$, while of the right side it is still $(k-1)(q-1)$. This proves the assertion of the theorem.
\end{proof}

For $q=2$ the above theorem is tight. It asserts that there are measurable $(t+1)$-colorings of a given necklace for which there is no fair $2$-splitting using at most $t$ axis-aligned hyperplane cuts. The set of such colorings is even dense. Due to a result of Goldberg and West \cite{GoWe85} it is the minimum number of colors. Together with Theorem \ref{lozi} of de Longueville and \v{Z}ivaljevi\'{c} this fully characterizes the set of tuples $(t_1,\dots,t_d)\in\N^d$, such that any $k$-colored $d$-dimensional necklace has a fair $2$-splitting using at most $t_i$ hyperplanes aligned to $i$-th axis for each $i$. Namely, it is the set $$\{(t_1,\dots,t_d)\in\N^d:\;k\leq t_1+\dots+t_d\}.$$

\section{Arbitrary Hyperplane Cuts}\label{SectionArbitraryHyperplaneCuts}

\begin{theorem}[Lasoń, \cite{La13c}]\label{TheoremNecklacesArbitrary}
For every integers $d,k,t\geq 1$ and $q\geq 2$ if $k(q-1)>dt+d+q-1$, then there is a measurable $k$-coloring of $\R^d$ such that no $d$-dimensional necklace has a fair $q$-splitting using at most $t$ arbitrary hyperplane cuts. The set of such colorings is dense. 
\end{theorem}

\begin{proof}
An arbitrary hyperplane (not necessarily axis-aligned) can be described by $d$ numbers $\beta^1,\dots,\beta^d$, which we will call \emph{parameters}. There are several ways of defining parameters, fortunately usually transition functions between different definitions are rational over $\Q$. For example for a hyperplane not passing through $O=(0,\dots,0)$ we can define parameters to be coefficients of its normalized equation $\beta^1x_1+\dots+\beta^dx_d=1$. In general pick $d+1$ rational points in general position $P_1,\dots,P_{d+1}$. Then no hyperplane contains all of them. We describe a hyperplane $H$ using coefficients of its normalized equation with $O$ being moved to a point $P_i$ which does not belong to $H$. We will make use of the following classical fact.

\begin{lemma}\label{rat} 
Volume of a bounded set in $\R^d$ which is an intersection of a finite number of half spaces is a rational function over $\Q$ of parameters of hyperplanes supporting these half spaces.
\end{lemma}

\begin{proof}
Such a set $S$ is a convex polytope. From Cramer's formula it follows that coordinates of vertices of $S$ are rational functions over $\Q$ of parameters. We make the barycentric subdivision of $S$ into simplices. Vertices of these simplices are barycenters of faces of $S$, hence they are also rational functions over $\Q$ of parameters.

Volume of each simplex equals to the determinant of the matrix of coordinates of its vertices divided by $d!$. Thus it is a rational function over $\Q$ of parameters. Then so is the volume of the original polytope.   
\end{proof}

Granularity of the splitting is the largest $g$ such that any piece of the splitting contains an axis-aligned cube with side of length $g$. We denote by $\mathcal{B}_n$ the set of $k$-colorings from $\mathcal{M}$ for which there exists at least one $d$-dimensional necklace contained in $[-n, n]^d$ which has a fair $q$-splitting with at most $t$ hyperplanes and granularity at least $1/n$. These are `bad' colorings. 

As in the previous Section \ref{SectionAxisAlignedHyperplaneCuts} our aim is to show that sets $\mathcal{B}_n$ are nowhere dense, provided suitable relation between $d,k,q$ and $t$ holds. This will imply that the set of colorings which we are looking for is dense.

The proof that sets $\mathcal{B}_n$ are closed in $\mathcal{M}$ goes exactly the same as the proof of Lemma \ref{LemmaClosed}. To show that if $(k-1)(q-1)>dt+d$, then every $\mathcal{B}_n$ has empty interior we repeat the proof of Lemma \ref{LemmaMain} and make small modifications. Instead of a single parameter $\beta_i$ of a hyperplane we have $d$ parameters $\beta_i^1,\dots,\beta_i^d$. Due to Lemma \ref{rat} the amount of color $j$ different from `white' contained in each part $l$ of the $q$-splitting is a rational function $R_{j,l}$ from 
$$\Q(\alpha_i,\beta_m^p,x_{j;i_1,\dots,i_d}:\;i=0,\dots,d;\;m=1,\dots,t;\;p=1,\dots,d;\;(i_1,\dots,i_d)),$$ 
instead of a polynomial $W_{j,l}$. Now the rank in the matroid $M_{\K}$ for the field
$$\K=\Q(x_{j;i_1,\dots,i_d}:\;(i_1,\dots,i_d)\neq (i^r_1,\dots,i^r_d)\; r=2,\dots,q)$$ 
of the set of left sides of equations analogous to $(\spadesuit)$ is at most $dt+d+1-1$, while of the set of right sides equals to $(k-1)(q-1)$. Hence the hypothesis that $\mathcal{B}_n$ has non-empty interior implies inequality $(k-1)(q-1)\leq dt+d$. It follows that if $(k-1)(q-1)>dt+d$, then sets $\mathcal{B}_n$ are nowhere dense. This proves the assertion of the theorem.
\end{proof}

Modifying slightly the above argument one may obtain a version for cuboids with inequality $k(q-1)>dt+2d+q-2$. Other modification, similar to the proof of Theorem \ref{TheoremNecklace}, leads the following theorem.

\begin{theorem}[Lasoń, \cite{La13c}]\label{TheoremNecklaceArbitrary}
For every integers $d,k,t\geq 1$ and $q\geq 2$ if $k(q-1)>dt+q-2$, then there is a measurable $k$-coloring of a $d$-dimensional necklace for which there is no fair $q$-splitting using at most $t$ arbitrary hyperplane cuts. The set of such colorings is dense. 
\end{theorem}

For $q=2$ we get a necessary condition ($k(q-1)\leq dt$) for mass equipartitioning problem. It is an interesting question if this condition is also sufficient for existence of a fair $q$-splitting of size at most $t$ of a $k$-colored $d$-dimensional necklace. For $(d,k,q,t)=(d,k,2^t,t)$ the question coincides with the conjecture of Ramos \cite{Ra96}. A small support is the positive answer for $(d,k,q,t)=(d,d,2^l,2^l-1)$, which follows directly from $2^l-1$ applications of the Ham Sandwich Theorem for measures \cite{Ma03}. 

\section{Colorings Distinguishing Cubes}\label{SectionColoringsDistinguishingCubes}

Here we are interested in a problem of a slightly different flavor. We say that a measurable coloring of $\R^d$ \emph{distinguishes axis-aligned cubes} if no two nontrivial axis-aligned cubes contain the same measure of every color. What is the minimum number of colors needed for such a coloring?

We get our results in a similar way to the proof of Theorem \ref{TheoremNecklaces}. For each $n$ we define a set $\mathcal{B}_n$ of all $k$-colorings from $\mathcal{M}$ for which there exist two $d$-dimensional cubes $A,B$ contained in $[-n, n]^d$ which have the same measure of every color, and such that $A\setminus B$ contains a translation of the cube $\left(0,\frac{1}{n}\right)^d$. The `bad' colorings are $\mathcal{B}=\bigcup_{n=1}^{\infty}\mathcal{B}_n$. Sets $\mathcal{B}_n$ are closed, and we can easily get an analog of Lemma \ref{LemmaMain}. This shows the following theorem.

\begin{theorem}[Lasoń, \cite{La13c}]\label{2d+3}
For every $d\geq 1$ there exists a measurable $(2d+3)$-coloring of $\R^d$ distinguishing axis-aligned cubes. The set of such colorings is dense. 
\end{theorem}

For $d=1$ this gives the second result of our previous paper \cite{AlGrLaMi09}, namely the existence of a measurable $5$-coloring of $\R$ with no two equally colored intervals. We suspect that this number of colors is optimal. 

\begin{question}
Fix $d\geq 1$. Does for every measurable $(2d+2)$-coloring of $\R^d$ there exist two nontrivial axis-aligned cubes which have the same measure of every color?
\end{question}

A slight modification of the argument leads to a version of Theorem \ref{2d+3} for cuboids with $4d+1$ colors.

Vre\'{c}ica and \v{Z}ivaljevi\'{c} \cite{VrZi13} solved the case of arbitrary continuous probability measures instead of measurable colorings of $\R^d$. They prove that $d+1$ measures are sufficient to distinguish axis-aligned cubes, and that it is tight.

\begin{theorem}[Vre\'{c}ica, \v{Z}ivaljevi\'{c}, \cite{VrZi13}]\label{vrec}
For every $d$ continuous probability measures $\mu_1,\dots,\mu_d$ on $\R^d$ there are two nontrivial axis-aligned cubes $C$ and $C'$ such that $\mu_i(C)=\mu_i(C')$ holds for every $i$. Moreover, there exist $d+1$ continuous probability measures on $\R^d$ distinguishing axis-aligned cubes.
\end{theorem}

In fact Vre\'{c}ica and \v{Z}ivaljevi\'{c} show that even one can find any finite number of axis-aligned cubes in $\R^d$ which are not distinguished by given $d$ measures. They also get results analogous to Theorem \ref{vrec} for cuboids in $\R^d$ with $2d-1$ colors. 

For sure, the minimum number of colors in the two cases (arbitrary continuous measures, and measurable colorings of $\R^d$) is different. One can show the following. 

\begin{example} 
There are two measures on $\R$ distinguishing intervals, while no measurable $2$-coloring of $\R$ distinguishes intervals.
\end{example}


\end{document}